\documentclass[12pt]{amsart}
\usepackage[pagewise]{lineno}
\usepackage{amssymb,amsthm,amsmath}
\usepackage{mathtools}
\usepackage{amsfonts}
\usepackage{amscd}
\usepackage{amssymb}
\usepackage{enumerate}
\usepackage{bbm}
\usepackage{graphicx}
\allowdisplaybreaks
\usepackage{mathabx}
\usepackage{color}
\usepackage{amsbsy}
\usepackage{graphicx}
\usepackage{amsthm}
\usepackage{amsmath}
\usepackage{amsxtra}
\usepackage{mathrsfs}
\usepackage{bbm}
\usepackage{dsfont}
\usepackage{enumitem}
\usepackage{comment}
\usepackage{relsize}
\usepackage{enumerate}
\usepackage{xcolor}
\usepackage{float} 
\usepackage{url}
\usepackage{soul}
\usepackage[hidelinks]{hyperref} 
\usepackage{ifthen}

\newtheorem{theorem}{Theorem}[section]
\newtheorem{lemma}[theorem]{Lemma}
\newtheorem{corollary}[theorem]{Corollary}
\newtheorem{proposition}[theorem]{Proposition}

\theoremstyle{definition}
\newtheorem{definition}[theorem]{Definition}
\newtheorem{remark} [theorem]{Remark}

\theoremstyle{remark}
\newtheorem{Remark}[theorem]{\rm \bf Remark}
\numberwithin{equation}{section}

\newcommand{\C}{\mathbb{C}}

\newcommand{\R}[1]{\mathbb{R}^{#1}}
\newcommand{\bN}{\mathbb{N}}

\newcommand{\fa}{\mathfrak{a}}
\newcommand{\fg}{\mathfrak{g}}
\newcommand{\fh}{\mathfrak{h}}

\newcommand{\fn}{\mathfrak{n}}

\newcommand{\fp}{\mathfrak{p}}

\newcommand{\fq}{\mathfrak{q}}

\newcommand{\aq}{\mathfrak{a}_q}
\newcommand{\aqc}{\mathfrak{a}_{q, \C}^\ast}

\newcommand{\cF}{\mathcal{F}}
\newcommand{\cW}{\mathcal{W}}
\newcommand{\cO}{\mathcal{O}}

\newcommand{\beas}{\begin{eqnarray*}}
\newcommand{\eeas}{\end{eqnarray*}}
\newcommand{\bes} {\begin{equation*}}
\newcommand{\ees} {\end{equation*}}
\newcommand{\be} {\begin{equation}}
\newcommand{\ee} {\end{equation}}
\newcommand{\bea} {\begin{eqnarray}}
\newcommand{\eea} {\end{eqnarray}}

\let\oldproofname=\proofname
\renewcommand{\proofname}{\rm\bf{\oldproofname}}

\newcommand{\ia}{i\mathfrak{a}^*}
\newcommand{\eE}{E^{\circ}}
\newcommand{\der}{\frac{d}{d\lambda}}
\renewcommand{\l}{\lambda}
\newcommand{\ti}[1]{\widetilde{#1}}
\newcommand{\mr}[1]{\mathrm{#1}}

\title{Boundedness of Eisenstein integrals on  split rank one semisimple symmetric spaces}
\author{Sanjoy Pusti and Iswarya Sitiraju}
\address{Sanjoy Pusti \endgraf Department of Mathematics, \endgraf INDIAN INSTITUTE OF TECHNOLOGY BOMBAY, \endgraf Powai, Mumbai-400076, India.}
\email{sanjoy@math.iitb.ac.in}

\address{Iswarya Sitiraju, \endgraf Department of Mathematics, \endgraf INDIAN INSTITUTE OF TECHNOLOGY BOMBAY, \endgraf Powai, Mumbai-400076, India.}
\email{iswarya@math.iitb.ac.in}


\begin{document}

\begin{abstract}
    We characterise the bounded left $K$-invariant normalized Eisenstein integrals on split rank one semisimple symmetric spaces. As a consequence we prove Hausdorff-Young inequality on these spaces. We also prove similar result for $K$-finite Eisenstein integrals on pseudo-Riemannian real hyperbolic spaces.
\end{abstract}

\maketitle

\section{Introduction}
Let $G$ be a connected semisimple Lie group with finite center and $K$ a maximal compact subgroup of $G$. Let $\phi_\l$ be the elementary spherical functions on the Riemannian symmetric space $G/K$, with spectral parameter $\l$ in the complexified dual space denoted by $\fa^*_{\C{}}$.
A celebrated theorem of Helgason-Johnson (\cite{HJ69}) states that there exists a tube domain $S_1$ in the complexified dual space $\fa^*_{\C{}}$ such that $\phi_\l$ is bounded on $G/K$ if and only $\l$ belongs to $S_1$. Moreover, $|\phi_\l(x)| \leq 1$ for $\l \in S_1$ and for all $x\in G$. In other words, the theorem characterises the bounded elementary spherical functions on $G/K$. As a consequence, it follows that  for a left $K$-invariant integrable function $f$ on $G/K$, the spherical Fourier transform $\cF f(\l)$ is  holomorphic in the interior $S_1^\circ$ and  continuous on $S_1$. Also, for $f\in L^1(K\backslash G/K)$, we have
\[|\cF f(\l)| \leq \|f\|_1, \qquad \l \in S_1.\]

This result was generalized to $L^r$ functions in \cite{EK82} for $1< r<2$.
A generalized Hausdorff-Young inequality and a generalized Riemann Lebesgue lemma has been shown to hold in  $S_r^\circ$ in \cite{EK82} and a variant of Hausdorff-Young inequality for spherical functions has been shown in \cite{CM93}. Similarly, a Hausdorff-Young inequality was  proved for $K$-finite $L^p$ functions on $G/K$ in the domain $S_r$ in \cite{EKS87}. Analogue of this Helgason-Johnson theorem is proved for the Jacobi functions $\varphi_{-i\l}^{(\alpha, \beta)}$ for $\alpha\geq\beta\geq -\frac{1}{2}$ by Flensted-Jensen and Koornwinder (\cite{FT73}).

\vspace{.4cm} 

Now, consider the triple $(\fa,\Sigma, m)$, where $\fa$ is a finite dimensional real Euclidean vector space, a root system $\Sigma$ and a multiplicity function $m$ on $\Sigma$. Let $\phi_\l$ be the Heckmann-Opdam hypergeometric functions associated with the triple $(\fa, \Sigma, m)$. The Helgason-Johnson theorem has been extended to this case (in \cite{NPP14}). In this case, we have
 a tube $S_1$ in the complexified dual vector space $\fa^*_{\C{}}$ such that,  the function $\phi_\l$ is bounded on $\fa$ if and only if $\l \in S_1$. Moreover, the Fourier transform satisfies the generalized Riemann-Lebesgue lemma on $S_1$ and also a Hausdorff-Young inequality holds on a tube domain $S_r$ depending on $r$ with $1 \leq r \leq 2$. If the triple $(\fa, \Sigma, m)$ is associated with the Riemannian symmetric space $G/K$, then the the function $\phi_\l$ coincides with the elementary spherical function on $G/K$ and thus, the domain in which they are bounded is the same. This result also proved for certain type of Heckman-Opdam hypergeometric functions for the root system BC (in \cite{NP22}).
\vspace{.5cm} 

 Let $G$ be a real connected semisimple linear Lie group and $H$ be a connected subgroup such that $(G,H)$ is a symmetric pair. Let $K$ be the maximal compact subgroup of $G$.  The Eisenstein integrals are defined as the  matrix coefficients of $K$-finite vectors with $H$-fixed vectors of parabolically induced representations. These functions were first introduced in \cite{vdB92} and also studied in \cite{VS97,VS97-2}. The left $K$ invariant Eisenstein integrals are the generalizations of the elementary spherical functions on the Riemannian symmetric spaces that is, when $H=K$ (cf. \cite{HC58}) and the spherical functions introduced by Oshima and Sekiguchi for the $K_\epsilon$ symmetric spaces introduced in \cite{OS80}. \\

Let $(\mathfrak{a}_q,\Sigma,m)$ be the root system associated with $(G,H,K)$ and $W$ be the Weyl group corresponding to the root system $(\mathfrak{a}_q,\Sigma)$.  Let $\cW$ be a specific subgroup of the Weyl group $W$. For a vector $\eta$ in $\C^{\cW}$ and the spectral parameter $\l \in \mathfrak{a}_{q, \C}^*$, the Eisenstein integrals are denoted by $E(\l,\eta)(\cdot)$. For $x \in G/H$, the function $\l \mapsto E(\l,\eta)(x)$ can be meromorphic in the complexified dual space $\fa_{q,\C}^*$. Hence, $E$ is normalized using their asymptotic behavior such that there are no singularities on $i\mathfrak{a}_q^*$ and is denoted as $\eE$.  For the Riemannian case , the Eisenstein integral is the elementary spherical function $\phi_\lambda$ and the normalized Eisenstein integral is $\phi_\l/c(\l)$, where $c(\l)$ is the Harish-Chandra c-functions.
In this article, we will study the normalized Eisenstein integrals on the split rank one semisimple symmetric spaces.
\vspace{.4cm}

  Let $x_0 = eH$ be the base point of $G/H$. For a fixed $w \in \cW$ and $\eta \in \C^{\cW}$, the normalised Eisenstein integral $\eE_w(\lambda,\eta)(t) := \eE_w(\lambda,\eta)(a_t w\cdot x_0)$  
is given by 
\begin{equation}\label{eq:eisen1}
    \eE_w(\lambda,\eta)(t) = \eta_w 2^{\lambda-\rho}\frac{\Gamma\left(\frac{\rho+\lambda}{2}\right)\Gamma\left(\frac{-\rho+\lambda+m_1^++m_2^++1}{2}\right)}{\Gamma(\lambda)\Gamma\left(\frac{m_1^++m_2^++1}{2}\right)} {}_2F_1\left(\frac{\rho+\lambda}{2}, \frac{\rho-\lambda}{2}; \frac{m_1^++m_2^++1}{2}; -\sinh ^2 t\right).
\end{equation}

 This normalized Eisenstein integrals $\eE_w(\l,\eta)$ are not entire functions. In fact, for each $R>0$, the functions $\lambda\mapsto \eE_w(\l,\eta)(t)$ has finitely many simple poles in $-a^\ast(R)$, where 
 \bes
a^\ast(R)=\{\lambda\in \mathbb C \mid \Re\lambda \leq R\}.
\ees
Moreover, there exists a polynomial $p_R$ (which is a finite product of linear polynomials) such that the function $\lambda\mapsto p_R(\lambda)\eE_w(\l,\eta)(t)$ is holomorphic on $-a^\ast(R)$. Then we characterise all those $\lambda\in -a^\ast(R)$, for which the function $t\mapsto p_R(\lambda)\eE_w(\l,\eta)(t)$ is bounded on $(0,\infty)$ (see Theorem \ref{thm: HJO} in Section 4).

\vspace{.4cm}

Using \cite[Lemma 2.1]{FT73} it follows that for $m_1^+\geq m_1^- + 2m_2^-$,  the function 
\bes 
t\mapsto {}_2F_1\Big(\frac{\lambda+\rho}{2}, \frac{-\lambda +\rho}{2};\frac{m_1^++m_2^++1}{2}; -\sinh^2{t}\Big)
\ees
is bounded on $(0, \infty)$ if and only if $|\Re\lambda|\leq \rho$. Thus, this result and (\ref{eq:eisen1}) gives characterisations of all $\lambda$ for which $t\mapsto p_R(\lambda)\eE_w(\l,\eta)(t)$ is bounded on $(0,\infty)$, only for the case when $m_1^+\geq m_1^- + 2m_2^-$. But for the case $m_1^+< m_1^- + 2m_2^-$, we will not be able to use the result of Flensted-Jensen and Koornwinder \cite{FT73}. In this case,  we mainly adapt the methods of \cite{NPP14}. First we get a series expansion of the normalized Eisenstein integrals at any $\lambda$ and using this we characterise all the functions  $t\mapsto p_R(\lambda)\eE_w(\l,\eta)(t)$ that are bounded.

\vspace{.5cm} 

Then as an application of the theorem we proved Hausdorff-Young inequality (see Theorem \ref{thm:HY} in section $5$) for this case. Also, we obtain similar (partial) results for the left $K$-finite normalized Eisenstein integral (see Theorem \ref{thm:ktypehj} and Theorem \ref{thm:HYktype} in section $6$) for the (pseudo-Riemannian) real hyperbolic spaces $SO_e(p,q)/SO_e(p-1,q)$.
We note that $\mathcal W=\{1\}$ for the case when $q>1$ and $\mathcal W=\{\pm 1\}$ for the case when $q=1$. Therefore, there is only one orbit of $\sigma$-minimal parabolic subgroup on the (pseudo-Riemannian) real hyperbolic spaces for $q>1$ and two disjoint orbits of $\sigma$-minimal parabolic subgroup for $q=1$. This makes the structure and analysis different for these two cases.
We will provide the required preliminaries of these cases separately in section $6$ and then unify them to prove the results. 

\vspace{.4cm}

We organise the paper as follows: In the preliminaries section, we first give necessary preliminaries for the general semisimple symmetric spaces $G/H$ and then give preliminaries for the split rank one semisimple symmetric spaces. In the next section we deduce explicit formula of the Harish-Chandra $c$-function.  In section $4$, we prove the analogue of Helgason-Johnson theorem on split rank one semisimple symmetric spaces. In section $5$, we prove Hausdorff-Young inequality and Riemann-Lebesgue lemma. Finally, in section $6$, we prove similar results for fixed left $K$-type normalized Eisenstein integrals on pseudo-Riemannian real hyperbolic spaces $\mathrm{\mathrm{SO}}_e(p,q)/\mathrm{\mathrm{SO}}_e(p-1, q)$.

\section{Preliminaries}
Throughout the article we shall use the standard notations $\mathbb N, \mathbb N_0,\mathbb Z,  \mathbb R, \mathbb C$ for Natural numbers, nonnegative integers, integers, real numbers and complex numbers, respectively. Given two
nonnegative functions $f$ and $g$ on a domain D, we write $f \asymp g$ if there exists positive
constants $C_1$ and $C_2$ so that $C_1g(x) \leq f(x) \leq C_2g(x)$ for all $x \in D$.

\vspace{.4cm}
In this section, we describe the necessary preliminaries regarding  harmonic analysis on semisimple symmetric spaces. These are standard and can be found, for example, in \cite{A01, VS97, VS97-2,  HS94, R79}. To make the article self-contained, we shall gather only those results which will be used throughout this paper.

\vspace{.4cm}

Let $G$ be a connected semisimple linear Lie group, $H$ a connected subgroup of $G$ and $K$ be a maximal compact subgroup of $G$. Let $\theta$ be the Cartan involution and $\sigma$ be a involution such that $H=G^\sigma_e,$ where $G^\sigma$ is the subgroup of fixed points for $\sigma$ and $G^\sigma_e$ is the identity component of $G^\sigma$. We  assume that $\sigma\theta=\theta\sigma$. Let \bes \mathfrak{g}=\mathfrak{h}\bigoplus \mathfrak{q}, \ees be the decomposition of $\mathfrak{g}$ induced by $\sigma$.  Also let \bes \fg=\mathfrak{k}\oplus\fp, \ees be the Cartan decomposition.   Let \bes \fg_+=\mathfrak{k}\cap\fh \bigoplus \fp\cap\fq \text{ and } \fg_-=\mathfrak{k}\cap \fq \bigoplus \fp\cap \fh.\ees Then $\fg=\fg_+ \bigoplus \fg_-$ and $G_+=G^{\sigma\theta}_e$ has Lie algebra $\fg_+$ with maximal compact subgroup $K\cap H$. Let $\aq$ be the maximal abelian subspace of $\fp\cap\fq$ and $A_q=\exp{\aq}$. Let $\Sigma(\aq, \fg_+)$ be the corresponding set of restricted roots and 
let $W_{K\cap H}$ be the Weyl group associated with the root system $\Sigma(\aq, \fg_+)$. Let $\aq^+$ be the corresponding positive Weyl chamber and $A_q^+=\exp {\aq^+}$.
Let $\mathfrak{a}$ be a maximal abelian subspace of $\fp$ containing $\aq$, then $\mathfrak{a}\cap \fq=\aq$. We denote the real dual of $\fa_q$ by $\fa_q^*$.  
For $\alpha\in\aq^\ast$, we define \bes
\fg_\alpha=\left\{Y\in\fg\mid [H, Y]=\alpha(H)Y \text{ for all } H\in \aq\right\}.
\ees
Let $\Sigma=\Sigma(\aq, \fg)$ be the set of $\alpha\in \aq^\ast$ such that $\fg_\alpha\not=\{0\}$. We then have $\Sigma(\aq, \fg_+)\subseteq \Sigma(\aq, \fg)$. 
Let $\fg_\alpha^{\pm}=\fg_\alpha\cap \fg_\pm$. 
Let \bes W=W(\aq, \fg)=N_K(\aq)/Z_K(\aq),\ees be the Weyl group of $\aq$ in $\fg$ associated with the root system $\Sigma(\aq, \fg)$. Let \bes \mathcal W=N_{K\cap H}(\aq)\backslash N_K(\aq)/Z_K(\aq)=W/W_{K\cap H}.\ees

 Let $m_\alpha = \mr{dim} \fg_\alpha$ be the multiplicity of $\alpha$ and let $m_\alpha^\pm=\dim \fg_\alpha^\pm$. Then $m_\alpha=m_\alpha^+ + m_\alpha^{-}$ and $m_\alpha^+$ is the multiplicity of $\alpha$ as a member of $\Sigma(\aq, \fg_+)$. Let \be\label{jacobian} J(a)=\prod_{\alpha\in \Sigma^+(\aq, \fg) } (a^\alpha-a^{-\alpha}) ^{m_\alpha^+}(a^\alpha+a^{-\alpha})  ^{m_\alpha^-},
\ee
for $a \in A_q$. We define \bes
\rho=\frac{1}{2}\sum_{\alpha\in \Sigma^+(\aq, \fg)}m_\alpha \alpha\in \aq^\ast.
\ees
Every element $g\in G$ has a decomposition as \be \label{decomposition} g=kah, k\in K, a\in A_q,  \text{ and } h\in H,
\ee and $a$ is unique upto $W_{K\cap H}$.
Consider the set \[X_+ = \cup_{w \in \cW}KA_q^+w H.\] When viewed as a subset of $G/H$, the set $X_+$ is open and dense. Then there exists a unique Haar measure on $A_q$ such that for $f \in L^1(X)$
\begin{equation}\label{eq:measure}
    \int_{G/H} f(x)dx = \sum_{w\in \cW}\int_K \int_{A_q^+}f(kawH)J(a)da dk,
\end{equation}
where $dk$ is the normalized Haar measure on $K$ (see \cite[Thm 3.9, p.24]{AO05}). 
We will denote the complexification of $\mathfrak{a}_{q}$ as $\mathfrak{a}_{q,\C{}}$ and the dual of $\fa_{q,\C{}}$ as $\fa_{q,\C{}}^*$.

The $\sigma$-minimal parabolic subgroup $P$ is a parabolic subgroup satisfying $\sigma \theta P = P$ and is minimal among all other parabolic subgroups satisfying this identity.  We then have the following theorem \cite[Thm. 13]{R79}:
\begin{theorem}
    Let $P$ be a $\sigma$-minimal parabolic subgroup of G and $\cW$ be as defined above.  Then
    \begin{enumerate}
        \item For each $\omega \in  \cW$, the orbit $\cO_\omega = P\omega H$ is open in $G$.
        \item The orbits $\cO_\omega$ are mutually disjoint.
        \item The disjoint union $\bigcup_{\omega \in \cW}\, \cO_{\omega}$ is dense in $G$.
    \end{enumerate}
\end{theorem}
Let $P = MAN$ be the Langlands decomposition. Let $\pi_{\xi, \l}= \mr{Ind}_{P}^G(\xi \otimes \l \otimes 1)$ be the parabolically induced representations on the space of continuous functions $C(K;\xi)$. The left K-invariant Eisenstein integrals, denoted by $E(\l)$, are linear combinations of matrix coefficients for the induced representations $\pi_{1,\l}$ with $\xi = 1$ and $\l \in \C$. Let $1_{-\l}$ be the $K$-fixed vector as a constant function $1$ on $K$. For $\eta \in \C^{\cW}$, $\l \in \C$ such that $\Re \l > \rho$ we define
\[j(\l,\eta) ( x) = \begin{cases}
    a^{\l-\rho}\eta_w &\text{if} \; x = man\,wh \in \cO_w\\
    0 &\text{otherwise},
\end{cases} \]
where $\eta_w$ is the $w$-th coordinate of $\eta$.
 Observe that $j(\l,\eta) \in C^{-\infty}(K;1)^H$ is an $H$-fixed distribution of $\pi_{\l,1}$ for $\l \in \C$ such that $\Re \l > \rho$. It was proven in \cite[Thm. 5.10, p.83]{vDB88} that
\begin{theorem}
    The map $\l \mapsto j(\l,\eta)$ can be extended as a meromorphic function on $\C$ for all $\eta \in \C^\cW$. Moreover, for almost all $\l$, the map $j(\l,\cdot) : \C^{\cW} \rightarrow C^{-\infty}(K,1)^H$ is a bijection.
\end{theorem}
The proof of this theorem is given in \cite[Thm. 5.10, p.83]{vDB88} and \cite{Olafsson1987}. Thus, the Eisenstein integrals are defined as 
\[E(\l,\eta)(x) = \langle 1_{-\l}, \pi_{\l,1}(x)j(\l,\eta)\rangle_{L^2(K)}.\]

The function $\l \mapsto E(\l,\eta)(x)$ may have singularities on $i\R{}$. Thus, we consider a normalization by their behavior at infinity which regularizes these functions on $i\R{}$. We denote these normalized Eisenstein integral as $\eE(\l, \cdot)(\cdot)$.

For $\lambda\in \fa^\ast_{q, \C}, \eta\in \C^{\mathcal W}$,
 Eisenstein integrals $E(\lambda,\eta)$ and normalized Eisenstein integrals  $E^\circ(\lambda,\eta)$ are left $K$-invariant and right $H$-invariant meromorphic functions in $\lambda$. These functions have the following properties:
\begin{enumerate}
\item These Eisenstein integrals and normalized Eisenstein integrals are $K$-invariant eigenfunctions for the invariant differential operators on $G/H$. 
More precisely, we have
    \be \label{diff-eins}
    DE(\lambda,\eta)=\gamma_q(D:\lambda)E(\lambda,\eta),
    \ee
    for all $D\in \mathbb D(G/H), \lambda\in \aqc$. Here $\gamma_q$ is the algebra homomorphism between $\mathbb D(G/H)$ and $\mathcal S(\aq)^W$.
      \item These are analytic functions on $G/H$.
    \item The functions $\l \mapsto E(\l , \eta)(x),E^\circ(\l , \eta)(x)$ are meromorphic on $ {\mathfrak{a}^\ast_{q, \C}}$, for fixed $\eta$ and $x$.
    \item The Eisenstein integral is normalized in such way that $E^\circ(\cdot , \eta)$ has no singularities on $i\aq^\ast$. 
     \item $\overline{E^\circ(-\overline{\lambda},\eta)(x)}=E^\circ(-\lambda,\overline{\eta})(x)$.
    \end{enumerate} 
Let $\mathfrak{a}_1$ be a maximal abelian subspace of $\fq$. The dimension of $\fa_1$ is called {\em rank} of the semisimple symmetric spaces $G/H$. 
 The following are the only rank one connected non-Riemannian symmetric spaces up to covering (for case by case treatment of Plancherel theorem see \cite{vD86, Molchanov1986} and references therein).
\begin{enumerate}
\item The real hyperbolic spaces: $\mathrm{\mathrm{SO}}_e(p, q)/\mathrm{\mathrm{SO}}_e(p-1, q), p>1, q>0$
\item The complex hyperbolic spaces: $\mathrm{SU}(p, q)/S\left(\mathrm{U}(1)\times\mathrm{U}(p-1, q)\right), p>1, q>0$.
\item The quaternion hyperbolic spaces: $\mathrm{Sp}(p, q)/\left(\mathrm{Sp}(1)\times\mathrm{Sp}(p-1, q)\right), p>1, q>0$.
\item The octonion hyperbolic space: $F_{4(-20)}/Spin(1, 8)$.
\item $\mathrm{SL}(n, \mathbb R)/\mathrm{GL}_+(n-1, \mathbb R)$, $n>1$.
\item $\mathrm{Sp}(n, \mathbb R)/\left(\mathrm{Sp}(1, \mathbb R)\times\mathrm{Sp}(n-1, \mathbb R)\right), n>1$.
\item $F_4(4)/Spin(4, 5)$.
\end{enumerate} 
The dimension of $\fa_q$ is called {\em split rank} of the semisimple symmetric spaces $G/H$. For a complete list of split rank one connected non-Riemannian symmetric spaces we refer to \cite[Table. II, p. 462]{OS84}.
\begin{Remark}
    In the case of Riemannian symmetric spaces, we have $H=K$, $\fp=\fq$ and hence $\fg_+=\fg, \,\,\fg_-=0$. This implies $\fa=\aq$  and $\Sigma(\aq, \fg)=\Sigma(\aq, \fg_+)$. So $W_{K\cap H}=W$ and hence $\mathcal W=\{1\}$. Also the Eisenstein integrals and normalized Eisenstein integrals becomes elementary spherical functions $\phi_\lambda$ and $c(\l)^{-1}\phi_\lambda$ respectively, where $c(\l)$ is the Harish-Chandra $c$-function. 
\end{Remark}
 For $R>0$, we define,
\bes
\aq^\ast(R)=\{\lambda\in \mathfrak{a}^\ast_{q,\C}\mid \Re\langle \lambda, \alpha\rangle \leq R \text{ for all } \alpha\in \Sigma^+(\aq, \fg) \}.
\ees
Let $S$ be the union of all the hyperplanes given by $\sigma_\mu=\{\lambda\in\mathfrak{a}_{q, \C}^*\mid \langle 2\lambda-\mu, \mu\rangle=0\}$, for $\mu\in \mathbb N_0(\Sigma^+(\aq, \fg))$. For $\lambda\not\in S$, there exists Harish-Chandra series
\begin{align*}
    \Phi_\lambda(a)=a^{\lambda-\rho}\sum_{\mu\in \mathbb N_0(\Sigma^+(\aq, \fg) )} a^{-\mu}\Gamma_\mu(\lambda),
\end{align*}
on $A_q^+$ with $\Gamma_\mu(\lambda)\in\C, \Gamma_0=1$ which solves the equation (\ref{diff-eins}).
Then the set \bes 
X_R=\{\mu\in \mathbb N_0(\Sigma^+(\aq, \fg))\mid \sigma_\mu\cap \aq^\ast(R)\not=\emptyset\},
\ees
is finite. 
Let $q_R(\lambda)$ be the polynomial \bes
q_R(\lambda)=\prod_{\mu\in X_R}\langle 2\lambda-\mu, \mu \rangle. 
\ees
Then $q_R(\lambda)\Phi_\lambda(a)$ is holomorphic as a function of $\lambda\in \aq^\ast(R)$. Moreover, for each $\epsilon>0$, there exists $M>0$ such that \bes
|q_R(\lambda) \Phi_\lambda(a)| \leq M (1 + |\lambda|)^c a^{|\Re\lambda|-\rho},
\ees
for all $a\in A_q$ with $\alpha(\log a) >\epsilon$ for all $\alpha\in \Sigma^+(\aq, \fg)$ and all $\lambda\in \aq^\ast(R)$.
Also \cite[Thm 7.6]{VS97} states that for each $s\in W$, there exists a unique {\em endomorphism valued} meromorphic function \bes
\lambda\mapsto C^\circ(s:\lambda)\in \text{ End}(\C^{\mathcal W}),
\ees
on $\aqc$ such that
\begin{equation}\label{eq: HCexpansion}
E^\circ(\lambda:\eta)(aw)=\mathlarger{\sum}_{s\in \mathcal W}\Phi_{s\lambda}(a) \left(C^\circ(s:\lambda)\eta\right)_w,
\end{equation}
for $a\in A_q^+, w\in \mathcal W, \eta\in \C^{\mathcal W}$, as a meromorphic identity in $\lambda\in \aqc$, where $(\cdot)_w$ is the $w$-th component in $\C{}^{\cW}$.
Also, 
\bes
C^\circ(s:\lambda)^\ast C^\circ(s:-\overline{\lambda})=Id,
\ees
for all $s\in W$ and $\lambda\in \aqc$. Moreover, we have  $C^\circ(1,\l) = Id$ and the following relation holds (see \cite[p. 121]{AO05}):
\[\eE(s\l, \eta) =\eE\left(\l, C^0(s\l, \eta)\right), \quad s\in W.\]

\vspace{.5cm}

 Let $\l \in \aqc$ with $\Re \l$ strictly dominant (that is, $\langle\Re \l, \alpha\rangle >0$ for all $\alpha \in \Sigma^+(\aq, \fg)$) and let $\omega \in \cW$. Then, we have the following asymptotics (see \cite[Part II, Prop. 7.7]{HS94}):
    \begin{equation}\label{eq:asym}
      a^{\rho-\lambda} \eE(\l, \eta)(aw) \rightarrow \eta_\omega, 
    \end{equation}
    as $a \rightarrow \infty$ in $A^+_q$ where $\eta_\omega$ is $\omega$-component of $\eta$.
   For $R>0$ there exists finitiely many roots $\{\alpha_i\}$ and real constants $\{c_i\}$ such that for all $\eta \in \mathbb{C}^\cW$ the singular set of the function $\l \mapsto \eE(\l,\eta)$ in $-a_q^*(R)$ is contained in
   \[\bigcup_{i=1}^k\{\l: \langle \l,\alpha_i\rangle = c_i\}.\]
Hence, the polynomial $p_R(\l) = \prod_{i=1}^k (\langle \l,\alpha_i\rangle - c_i)$ is such that $p_R(\lambda) E^\circ(\lambda,\eta)$ is holomorphic on a neighborhood of $-\aq^\ast(R)$, for all $\eta\in \C^{\mathcal W}$. The polynomial $p_R$ cancels all the the poles of $ E^\circ(\lambda,\eta)$ in $-a_q^*(R)$.
Moreover, the following estimate follows from \cite[Prop. 10.3 and Cor. 16.2]{vdB92} which says that 
\be \label{est: normalized-Eisenstein}
|p_R(\lambda)E^\circ(\lambda:\eta)(a)| \leq C (1 + |\lambda|)^N e^{(s + |\Re\lambda|)|\log a|} |\eta|,
\ee
for some constants $C, N$ and $s$, for all $\lambda\in -\aq^\ast(R), \eta\in\C^{\mathcal W}$ and $a\in A_q$.

For a $K$-invariant function $f$ on $G/H$, the Fourier transform of $f$ is defined as (\cite[Thm. 3.9]{AO05})
\bes
\mathcal F f(\lambda)(\eta)=\int_{X} f(x) E^\circ(-\lambda:\eta)(x)\,dx,
\ees
for $\eta\in \C^{\mathcal W}$. This $\mathcal F f(\lambda)$ is a linear form on $\C^{\mathcal W}$.

\subsection{\bf Split rank one cases:}
We will prove our theorems in the split rank one cases only. In this case the dimension of $\fa_q$ is one. The root system $\Sigma(\fa_q, \fg)$ is $\{\pm \alpha\}$ or $\{\pm \alpha, \pm 2\alpha\}$. Let $H_0\in \aq$ be a unique element such that $\alpha(H_0)=1$.  Let \bes m_1=m_{\alpha},m_1^+=m_{\alpha}^+, m_1^{-}=m_{\alpha}^{-} ;\,\, m_2=m_{2\alpha}, m_2^+=m_{2\alpha}^+, m_2^-=m_{2\alpha}^-.
\ees Then we have $m_1=m_1^+ + m_1^{-}$ and 
$m_2=m_2^+ + m_2^{-}$. Also the followings are true (in split rank one cases) \cite[p. 3]{S90}: 
\begin{enumerate}
    \item $m_1^+ + m_1^->0,$
    \item If $m_2^->0$ then $m_1^+=m_1^-$,
    \item $m_1^-=m_2^-=0$ if and only if $(\mathfrak{g}, \mathfrak{h}$) is Riemannian.
    \item If $m_1^+=m_2^-=0$ then $(\mathfrak{g}, \mathfrak{h}$) is of $\mathfrak{k}_\epsilon$ type.
\end{enumerate} Also $\rho$ is identified with $\rho(H_0)$ and \bes \rho=\frac{1}{2}(m_1 + 2m_2)= \frac{1}{2}(m_1^+ + m_1^{-} + 2m_2^{+} + 2m_2^{-}).\ees 
Also we have, $\cW = \{1\}$  or $\cW = \{\pm 1\}$. 
The Jacobian defined in (\ref{jacobian}) becomes 
\be \label{jacobian-1}
J(t) = (\sinh t)^{m_1^+}(\cosh t)^{m_1^-}(\sinh 2t)^{m_2^+}(\cosh 2t)^{m_2^-}.
\ee

Let $\Delta \in \mathbb{D}(G/H)$ be the Laplace Beltrami operator on $X$. 
The radial part of $\Delta$ is given by \cite[p. 8 and 9]{S90} and \cite[Thm. 2.1]{AM90}.
\begin{align*}
    L(\Delta) &= \frac{\partial^2}{\partial t^2} +(m_1^+ \coth t + m_1^- \tanh t + 2m_2^+\coth 2t +2m_2^- \tanh 2t) \frac{\partial}{\partial t}\\
    &= \frac{1}{J(t)}\frac{\partial}{\partial t}\left(J(t)\frac{\partial}{\partial t} \right).
\end{align*}
Let $x_0 = eH$ be the base point of $G/H$. For a fixed $w \in \cW$ and $\eta \in \C^{\cW}$, the Eisenstein integral $\eE_w(\lambda,\eta)(t) := \eE_w(\lambda,\eta)(a_t w\cdot x_0)$  
satisfies the following differential equation
\begin{equation}\label{eq: diffeq23}
    L(\Delta) \eE_w(\lambda, \eta)(t) = (\rho^2-\lambda^2)\eE_w(\lambda,\eta)(t).
\end{equation}

We now find explicit formula for $\eE_w(\l,\eta)(t)$ (cf. \cite[p. 38 and 39]{S90}).
If $m_2^-=0$, by substituting $s = -\sinh^2 t $, the above differential equation can be transformed into a hypergeometric differential equation and $\eE_w(\l,\eta)$ satisfies 
\[\left[-(1-s)s\frac{\partial^2}{\partial s^2} + \left(s(1+\rho) - \frac{1+m_1^++m_2^+}{2}\right)\frac{\partial}{\partial s}\right] f = \frac{\l^2-\rho^2}{4}f.\]
If $m_2^- >0$, then from \cite[(C.2), p. 32]{S90} we have $m_1^+ = m_1^-$. Thus by substituting $s = 2t$ and $r = -\sinh^2 s$ the differential equation (\ref{eq: diffeq23}) transforms into into a hypergeometric differential equation and $\eE_w(\l,\eta)$ satisfies
\[\left[-(1-r)r\frac{\partial^2}{\partial r^2} + \left(r\left(1+\frac{\rho}{2}\right) - \frac{1+m_1^++m_2^+}{2}\right)\frac{\partial}{\partial r}\right]f = \frac{\l^2-\rho^2}{16}f.\]

Since, $E_w$ is regular at $t=0$, for all $\l$ we have 
\[\eE_w(\lambda,\eta)(t) = \mr{const.} {}_2F_1\left(\frac{\rho+\lambda}{2}, \frac{\rho-\lambda}{2}; \frac{m_1^++m_2^++1}{2}; -\sinh ^2 t\right).\]

Using the following property (\cite[eq. 17, p. 63]{E81})

\begin{equation}\label{eq:hyperasymp}
\begin{split}
{}_2F_1(a, b ; c ; z) & =\frac{\Gamma(c) \Gamma(a-b)}{\Gamma(a) \Gamma(c-b)}(-z)^{-b} F\left(b, 1-c+b ; 1-a+b ; z^{-1}\right)  \\
& +\frac{\Gamma(c) \Gamma(b-a)}{\Gamma(b) \Gamma(c-a)}(-z)^{-a} F\left(a, 1-c+a ; 1-b+a ; z^{-1}\right),
\end{split}
\end{equation}
and 
\begin{equation}\label{eq:asym1}
    \lim_{t \rightarrow \infty} e^{(\rho-\lambda)t}\eE_w(\lambda,\eta)(t) = \eta_w, \quad \Re \lambda >0,
\end{equation}

we obtain for $\Re\lambda>0$,

\begin{equation}\label{eq:eisen}
    \eE_w(\lambda,\eta)(t) = \eta_w 2^{\lambda-\rho}\frac{\Gamma\left(\frac{\rho+\lambda}{2}\right)\Gamma\left(\frac{-\rho+\lambda+m_1^++m_2^++1}{2}\right)}{\Gamma(\lambda)\Gamma\left(\frac{m_1^++m_2^++1}{2}\right)} {}_2F_1\left(\frac{\rho+\lambda}{2}, \frac{\rho-\lambda}{2}; \frac{m_1^++m_2^++1}{2}; -\sinh ^2 t\right).
\end{equation}

Let \bes a^*(R)=\{\lambda\in\mathbb C\mid \Re\lambda\leq R\}.\ees
We note that $\aq^\ast(R)$ becomes $a^\ast(R)$ in this case. We will continue this ($a^\ast(R)$) notation  instead of $\aq^\ast(R)$ from now on.

From the expressions (\ref{eq:eisen}) of the normalized Eisenstein integrals we observe that, for a fixed $R>0$, there exists a polynomial $p_R$ such that $p_R(\l)\eE_w(\l,\eta)(t)$ is holomorphic in $-a^*(R)$. In fact, $p_R$ is given by 
\be\label{exp:p_R}
p_R(\lambda)=\prod_{\substack{k_1\in \bN_0, \\ \rho+2k_1\leq R}}(\lambda +\rho + 2k_1) \prod_{\substack{k_2\in \bN_0,\\ -\rho+ m_1^+ + m_2^+ +1+2k_1\leq R}}(\lambda -\rho +m_1^+ + m_2^+ +1 +2k_2).
\ee

\begin{proposition}\label{Prop:HCSERIES}
\begin{enumerate}
    \item There exists a unique absolutely convergent series for $\lambda \notin \frac{1}{2}\mathbb N_0$ for $t>0$,
    \be \label{HC-series}\Phi_\lambda(t) = e^{(\lambda-\rho)t}\sum_{m=0}^\infty e^{-mt}\Gamma_{m}(\lambda),
    \ee
with $\Gamma_{m}(\lambda) \in \C{}$ and $\Gamma_0=1$ which satisfies the differential equation
\[L(\Delta)\Phi_\lambda(t) = (\lambda^2-\rho^2)\Phi_\lambda(t).\]
    \item  Let $X_R= \{k \in \mathbb N_0: 0< k \leq R\}$. Then there exist a polynomial 
    \be \label{poly:qr}
    q_R(\lambda) = \Pi_{k \in X_R}k(2\lambda-k), 
    \ee
    such that $q_R \Phi_\lambda$ is holomorphic in $a^*(R)$.
    \end{enumerate}
    
\end{proposition}

We recall for a function $f \in C_c^\infty(K\backslash G/H)$, the Fourier transform of $f$ is defined by 
\begin{align*}
    \mathcal F f(\l)(\eta) &= \int_X f(x)\eE(-\l,\eta)(x)dx\\
    &=\sum_{w \in \cW} \int_{0}^\infty f(a_t w \cdot e_1)E^0(-\lambda, \eta)(a_tw \cdot e_1)J(t)dt,
\end{align*}
for all $\eta\in \C^{\cW}$ and $J(t)$ is Jacobian given in equation (\ref{jacobian-1}).
Therefore, $p_R(-\lambda)\cF f(\lambda)$ is holomorphic on $a^*(R)$ and \[\cF(-\lambda)(\eta) = \cF(\lambda)(C^0(-1,\lambda)\eta).\]
The following proposition is about the Fourier transform of compactly supported $K$-invariant functions on $X$ from \cite[Prop. 8.4]{HS94}.
\begin{proposition}\label{prop:paleyweiner}
    Let $\phi \in C_c^\infty(K\backslash G/H)$ and $n \in \mathbb{N}_{0}$. Then there exists a constant $M>0$ such that 
    \[|p_R(-\lambda)\cF\phi(\lambda)(\eta)| \leq M(1+|\lambda|)^{-n}e^{r|\mathrm{Re} \lambda|} \|\eta\|,\]
    for all $\lambda \in a^*(R)$ and $r$ depends on the size of the support of $\phi$. 
\end{proposition}

From \cite[Thm 7.1, Thm 7.4]{HS94} we have the following theorem: 
\begin{theorem} For $f\in L^2(X)$
  \begin{equation}\label{Plncherel}
    \left(\int_{\ia_q} |\cF(\l)(\eta)|^2d\l \right)^{1/2}\leq \|f\|_{L^2(X)} \|\eta\| ,
\end{equation}
where $d\l$ is the Lebesgue measure on $\ia_q$.   
\end{theorem}

\section{Harish-Chandra series and $C^0$-functions}
Let $G/H$ be a split rank $1$ semisimple symmetric spaces. In this section we will study Harish-Chandra series and will find explicit expression of $C^0$-functions.

We have the following results regarding Harish-Chandra series (see \cite[Prop 7.5]{HS94}, \cite[p. 66]{VS97}): 
 Let \be \label{modified-HC} 
 \Psi_\lambda(t) = \sqrt{J(t)}\Phi_\lambda(t) = e^{\lambda t}\sum_{m=0}^{\infty} \ti{\Gamma}_m(\lambda)e^{-mt}.
 \ee
The $\Psi_\lambda$ satisfies
\[\Psi_\lambda''(t) + d(t)\Psi_\lambda(t) = (\lambda^2-\rho^2)\Psi_\lambda(t),\]
where  \bes d(t) = J''/\sqrt{J} = \sum_{m=0}^\infty d_m e^{-mt}, \ees and $d_0 = \rho^2$.

By substituting (\ref{modified-HC}) into above we obtain the following recursive formula:
\bes
(\lambda^2 -\rho^2)\ti{\Gamma}_m=(\lambda-m)^2\ti{\Gamma}_m  - \sum_{i=0}^m \ti{\Gamma}_i d_{m-i}. \ees 
That is, 
    \bes
m(2\l-m)\ti{\Gamma}_m=\rho^2\ti{\Gamma}_m  - \sum_{i=0}^m \ti{\Gamma}_i d_{m-i},
\ees
which implies that
\[  m(2\l-m)\ti{\Gamma}_m= - \sum_{i=0}^{m-1} \ti{\Gamma}_i d_{m-i}.\]

Therefore, the poles of $\ti{\Gamma}_m$ lies in the set $\{1/2,1,..., m/2\}$. We also have the following relation: 
\[\Gamma_m(\lambda) = \sum_{i=0}^{m}b_i\ti{\Gamma}_{m-i}(\lambda),\]
where \[(J(t))^{-1/2} = e^{-\rho t}\sum_{m=0}^\infty b_m e^{-mt}.\]
We thus obtain that the poles of $\Gamma_m \subseteq \{1/2,1,..., m/2\}$.
Further we have the following (\cite[Theorem 7.4]{VS97})

\begin{theorem}\label{thm:Gamma}
For $R > 0$, there exists a polynomial $q_R$ and constants $M, \chi >0$ depending on $R$, such that
\[|q_R(\lambda)\Gamma_m(\lambda)| \leq M(1+m)^{\chi}(1+|\lambda|)^{\mathrm{\text{deg}} \,q_R},\]
for all $m \geq 0$ and $\lambda \in \overline{a}^*(R)$.
\end{theorem}
The polynomial $q_R$ is same as defined in (\ref{poly:qr}) is of the form \bes q_R(\l)=d (\l-\frac 12)(\l-1)\cdots (\l-\left \lfloor R/2 \right \rfloor),\ees where $d$ is a fixed constant. Therefore if $\l_0\in \{\frac{1}{2},1,..., \frac{m}{2}\}$, then \be\label{est:gamma-1}
|(\l-\l_0)\Gamma_m(\l)|\leq M(1+ m)^\chi,
\ee
for all $\l$ in a small neighborhood of $\l_0$. 

Next we give explicit expression of  $C^{0}(-1, \lambda)$. Our method is based on the idea of \cite{S90}.
\begin{lemma}\label{lem:c-function}
\begin{enumerate}
    \item If $\cW = \{1\}$, then $C^{0}(-1,\lambda)=c(\lambda)$ is a complex number given by
    \begin{equation}\label{eq:c-function}
    c(\lambda)=2^{2\lambda} \frac{\Gamma\left(\frac{\rho+\lambda}{2}\right)\Gamma(-\lambda)\Gamma\left(\frac{-\rho+\lambda+m_1^++m_2^++1}{2}\right)}{\Gamma\left(\frac{\rho-\lambda}{2}\right)\Gamma(\lambda)\Gamma\left(\frac{-\rho-\lambda+m_1^++m_2^++1}{2}\right)}.
\end{equation}

\item If $\cW = \{\pm 1\}$, then $C^{0}(-1,\lambda)$ is $2\times 2$ matrix given by
$C^{0}(-1,\lambda)=c(\lambda)I$ where
\[c(\l) = 2^{2\lambda}\frac{\Gamma((\rho+\lambda)/2)\Gamma(-\lambda)\Gamma((\lambda-\rho+1+m_1^+ + m_2^+)/2)}{\Gamma((\rho-\lambda)/2)\Gamma(\lambda)\Gamma((-\rho - \lambda + 1+m_1^+ + m_2^+)/2)},\] and $I$ is the $2\times 2$ identity matrix.

\end{enumerate}
    
\end{lemma}

\begin{proof}
Let $\cW = \{1\}$. Then in this case, $ C^{0}(-1,\lambda)=c(\lambda)$ is a complex number. Without loss of generality we assume that $\eta =1$. For {$\Re \lambda < 0$, using (\ref{eq:hyperasymp}), (\ref{eq:eisen}), (\ref{eq: HCexpansion}) and 
\begin{equation*}
    \lim_{t \rightarrow \infty}e^{(\rho+\lambda)t} \eE_1(\lambda,1)(t) = \lim_{t \rightarrow \infty}e^{(\rho+\lambda)t}[\Phi_\lambda(t) + \Phi_{-\lambda}(t)c(\lambda)]=  c(\l)
\end{equation*}}
we obtain 
\begin{equation*}
    c(\lambda)=2^{2\lambda} \frac{\Gamma\left(\frac{\rho+\lambda}{2}\right)\Gamma(-\lambda)\Gamma\left(\frac{-\rho+\lambda+m_1^++m_2^++1}{2}\right)}{\Gamma\left(\frac{\rho-\lambda}{2}\right)\Gamma(\lambda)\Gamma\left(\frac{-\rho-\lambda+m_1^++m_2^++1}{2}\right)}.
\end{equation*}

If $\cW = \{\pm 1\}$.  With respect to the basis $e_1$ and $e_2$, let $C^0(-1,\l) = [c_{ij}(\l)]$. Then for $t>0$,
\begin{align*}
    \eE_1(\l,e_1)(t) &= \Phi_\l(t) (1) + \Phi_{-\l}(t)c_{11}(\l) \\
    \eE_{-1}(\l,e_1)(t) &= \Phi_\l(t) (0)+ \Phi_{-\l}(t)c_{21}(\l) \\
    \eE_1(\l,e_2)(t) &= \Phi_\l(t)(0) + \Phi_{-\l}(t)c_{12}(\l) \\
    \eE_{-1}(\l,e_2)(t) &= \Phi_\l(t)(1) + \Phi_{-\l}(t)c_{22}(\l). 
\end{align*}

Thus, for $\Re \l <0$ we obtain from the expansion (\ref{eq: HCexpansion}) that
\[ \lim_{t \rightarrow \infty}e^{(\rho + \l)t}\eE_w(\l, \eta)(t) = [C^0(-1,\l)\eta]_w.\]
Identify that if $w=1$, then $i=1$ and if $w=-1$, then $i=2$. Then the above says that,
\begin{align*}
  c_{ij} &=   \lim_{t \rightarrow \infty}e^{(\rho + \l)t}\eE_w(\l, e_j)(t) \\
  &= \delta_{ij} \frac{2^{\l-\rho}\Gamma((\l+\rho)/2)\Gamma((\l-\rho+1)/2)}{\Gamma(\l)\Gamma(1/2)}{}_2F_1\left(\frac{\rho+\l}{2}, \frac{\rho-\l}{2}; \frac{1}{2}; -\sinh^2 t\right).
\end{align*}
Immediately, we obtain that $c_{ij} = 0$ if $i \neq j$.
Now, substituting  (\ref{eq:hyperasymp}) in the above we obtain the entries $c_{11}$ and $c_{22}$ as follows:

\begin{equation}\label{eq:C-matrix}
        \begin{split}
            c_{11}(\l) &= c_{22}(\l) = 2^{2\lambda}\frac{\Gamma((\rho+\lambda)/2)\Gamma(-\lambda)\Gamma((\lambda-\rho+1+m_1^+ + m_2^+)/2)}{\Gamma((\rho-\lambda)/2)\Gamma(\lambda)\Gamma((-\rho - \lambda + 1+m_1^+ + m_2^+)/2)},\\
            c_{12}(\l) &= c_{21}(\l) = 0.
        \end{split}
    \end{equation}

Thus, $C^0(-1,\l) = c(\l)Id$ where 
\[c(\l) = 2^{2\lambda}\frac{\Gamma((\rho+\lambda)/2)\Gamma(-\lambda)\Gamma((\lambda-\rho+1+m_1^+ + m_2^+)/2)}{\Gamma((\rho-\lambda)/2)\Gamma(\lambda)\Gamma((-\rho - \lambda + 1+m_1^+ + m_2^+)/2)}.\]
\end{proof}
We observe that $C^0(-1,\l)$ is a diagonal matrix with same diagonal entries $c(\lambda)$. Therefore,  we can write (\ref{eq: HCexpansion}) as
 \begin{equation}\label{eq:HCcombined}
     \eE_w(\l,\eta)(t) = \Phi_\l(t)\eta_w + c(\l)\eta_w\Phi_{-\l}(t).
 \end{equation}

 Since $E^0_w(-\lambda,\eta)(t) =  E^0_w(\lambda,C^0(-1,\l)\eta)$,   we obtain
 \begin{equation}\label{eq: Re<0}
     E^0_w(\lambda,\eta)(t) =\eta_w 2^{\lambda-\rho}\frac{\Gamma(\frac{\lambda+\rho}{2})\Gamma(\frac{\lambda-\rho+m_1^+ + m_2^+ +1}{2})}{\Gamma(\lambda)\Gamma(\frac{m_1^+ + m_2^+ +1}{2})}{}_2F_1\Big(\frac{\lambda+\rho}{2}, \frac{-\lambda +\rho}{2};\frac{q}{2}; -\sinh^2{t}\Big), 
 \end{equation}
for $\Re \lambda < 0$.
Thus, for $\Re \lambda < 0$ the singularities are contained in
\[ \{-\rho-2k_1: k_1 \geq 0 \} \cup \{\rho- m_1^+ - m_2^+ -1 -2k_2 : k_2 \geq 0\; \text{and}\; \rho- m_1^+ - m_2^+ -1 -2k_2 <0 \}. \]

From the previous results we have the following:

\begin{lemma}\label{lemm: pole-zero-description}The following are true :
\begin{enumerate}
    \item The singularities of $\Gamma_m$ are at most simple poles and contained in $\{\frac 12,1,...,\frac m2\}$.
    \item The singularities of $\Phi_\lambda$ (as a function of $\lambda$) are at most simple poles and contained in  $\frac{1}{2}\mathbb{N}_{0}$.
    \item The poles of $c(\l)$ are simple and equal to $\{-\rho-2k_1:k_1 \in\mathbb N_0\} \cup \{\rho - 1-m_1^+-m_2^+ -2k_2: k_2 \in\mathbb N_0\}\cup\{k_3:k_3\in\mathbb N\}.$
    \item The poles of $E^\circ_w(\cdot, \eta)(t)$ are simple and equal to $\{-\rho-2k_1:k_1 \in\mathbb N_0\} \cup \{\rho - 1-m_1^+-m_2^+ -2k_2: k_2\in\mathbb N_0\}$.
    \item The zeros of $c(\l)$ is equal to $\{\rho +2k_1:k_1 \in\mathbb N_0\} \cup \{-k_2: k_2\in\mathbb N\}\cup \{-\rho + 1 +m_1^+ + m_2^+ +2k_3: k_3 \in\mathbb N_0\}$.
    \item The zeros of $E^\circ_w(\cdot, \eta)(t)$ is equal to $\{-k : k \in \bN_0\}$.
\end{enumerate}
\end{lemma}

\begin{Remark}
    Suppose there exists a $k_2\in\mathbb N_0$ such that $\rho -1-m_1^+-m_2^+-2k_2 = 0$. Then from $(4)$ (in Lemma \ref{lemm: pole-zero-description}) above, it follows that $0$ is a pole $E^\circ_w(\cdot, \eta)(t)$ but we note from $(6)$ (in Lemma \ref{lemm: pole-zero-description}) above that $0$ also a zero of $E^\circ_w(\cdot, \eta)(t)$. Hence, the function $E^\circ_w(\cdot, \eta)(t)$ is holomorphic at $0$ and hence holomorphic on the line $i\mathbb R$. 
\end{Remark}

The Fourier Inversion formula for the general rank cases is discussed in \cite[Thm. 7.1]{vS99} and the explicit formula given below for the split rank one symmetric spaces is given in \cite[Eq. 5-13, p.128]{AO05}. 
    
\begin{theorem}[Fourier Inversion]\label{thm:Finversion}

    For $f \in C_c^\infty(K\backslash G/H)$ the Fourier inversion formula is given by
    \begin{equation}
        f(a_tw\cdot x_0) = \int_{i\R{}} \eE_w(\lambda, \mathcal F f(\lambda))(t) d\lambda + 4\pi i\sum_{\mu_k \in L}  \mr{Res}_{\l = -\mu_k}[\mathcal F f (\l) \Phi_{\l} (t)]_w ,
    \end{equation}
 where $L = \{\mu_k = \rho -1-m_1^+-m_2^+ -2k: k\in \bN_0\; \text{and}\; {\mu_k > 0}\}$.
\end{theorem}


\section{Helgason-Johnson's theorem}\label{sec:hj} 
In this section we prove the analogue of Helgason-Johnson theorem for split rank one semisimple symmetric spaces. Let \[S_1=\{\l \in \C: |\Re \l| \leq \rho\}.\] 
  
Let $\alpha, \beta, \l \in \C$ with $\alpha \neq -1,-2,\ldots$ and $\rho = \alpha+\beta +1$. From \cite[Eq. 2.3]{K75} the Jacobi functions with paramaters, $\alpha, \beta$ and $\l$ is defined as
\[\varphi_{-i\l}^{(\alpha,\beta)}(t) = {}_2F_1\left(\frac{\rho+\l}{2},\frac{\rho-\l}{2};\alpha +1; -\sinh^2t\right).\]

Then the normalized Eisenstein integral can also be written in the form of Jacobi functions for $\Re \l >0$ with parameters $\alpha = (m_1^+ + m_2^+ +1)/2 -1$, $\beta = (m_1^- + m_2^+ + 2m_2^- +1)/2-1$ and $\rho = \alpha+\beta+1$  as 
\[\eE_w(\l,\eta)(t) = \eta_w2^{\lambda-\rho}\frac{\Gamma(\frac{\lambda+\rho}{2})\Gamma(\frac{\lambda-\rho+m_1^+ + m_2^+ +1}{2})}{\Gamma(\lambda)\Gamma(\frac{m_1^+ + m_2^+ +1}{2})} \varphi_{-i\l}^{(\alpha,\beta)}(t).\]

Thus, from \cite[Lemma 2.3]{K75} and properties of Gamma function, we obtain the following estimate for each fixed $R>0$:
\begin{equation}\label{est:Jacobi}
    |p_R(\l)\eE_w(\l,\eta)(t)| \lesssim \|\eta\|(1+|\l|)^d(1+t)e^{(|\Re \l| -\rho)t},
\end{equation}

for $d =\text{deg}\, p_R$ and $\lambda\in -a^\ast(R)$.
Therefore, for each fixed $\eta\in \mathbb C^{\mathcal W}$, \bes p_R(\l)\eE_w(\l, \eta)(\cdot),
\ees
is bounded if $\l \in -a^*(R)$ and $|\Re \l| < \rho$. It is also known that (see \cite[Lemma 2.1]{FT73}) $\varphi_{-i\l}^{(\alpha,\beta)}(t)$ is bounded in $t$ if and only if $\lambda \in S_1$ for $\alpha\geq \beta\geq -\frac{1}{2}$. This gives characterisation of all $\lambda$ for which $t\mapsto p_R(\l)\eE_w(\l, \eta)(t)$ is bounded on $(0, \infty)$, only for the case when $m_1^+\geq m_1^- + 2m_2^-$. In this section we prove a similar result for general $m_1, m_2$.  

Moreover, we observe that from the asymptotic behavior (\ref{eq:asym1}), if $\Re \l >0$ and $\l$ not a pole of $\eE_w$  we have
\[\lim_{t\rightarrow\infty}|e^{(\rho-\l)t}\eE_w(\l,\eta)(t) -\eta_w| =0.\]
This implies that
\[|\eE_w(\l,\eta)(t)| \asymp |\eta_w|e^{(\Re \l -\rho)t}. \]
Thus, for the case when  $m_1^- + 2m_2^-\leq m_1^+ +2$, where there are no poles of $\eE_w(\cdot,\eta)(t)$ for $\Re \l >0$,  the characterization of boundedness (for $\Re \l >0$) also follows naturally from this asymptotic behavior. In particular,  $\eE_w$ is bounded if $0 <\Re \l \leq \rho$ and unbounded if $\Re \l > \rho$. However, we have the the main difficulties in the other cases. 

\begin{theorem} \label{thm: HJO}
  Fix $R>0$, $w \in \cW$ and $\eta \in \C^{\cW}$.  \begin{enumerate}
      \item Let $\l_0 \in (S_1 \cap-a^*(R)) \cup -\mathbb N_0$. Then the function $t\mapsto p_R(\l_0)\eE_w(\l_0,\eta)(t)$ is bounded on $(0,\infty)$.   
      \item Conversely, for $\l_0 \in -a^*(R)$ and for non-zero $\eta_w$, if the function $t\mapsto p_R(\l_0)\eE_w(\l_0,\eta)(t)$ is bounded on $(0,\infty)$, then $\l_0 \in S_1  \cup -\mathbb N_0$. 
  \end{enumerate} 
  
\end{theorem}
\begin{remark}
 From this theorem above, we conclude that for a fixed $R > \rho$, a non-zero $\eta_w$ and $\lambda_0\in -a^*(R)$, the function $p_R(\l_0)\eE_w(\l_0,\eta)(\cdot)$ is bounded on $(0,\infty)$ if and only if $\l_0 \in S_1 \cup -\bN_0$.
\end{remark}

\vspace{.4cm}
\begin{proof}[Proof of Theorem~\ref{thm: HJO}]

Throughout the proof we will assume that $\eta_w \neq 0$. 

{\bf Case 1:  $\Re \l_0 \geq 0$:}
For a fixed $R>0$, let $p_R$ be the polynomial given in (\ref{exp:p_R}).  Let  \bes F_w(\l,\eta)(t) = p_R(\l)\eE_w(\l,\eta)(t).\ees This is holomorphic at $\lambda_0$.
Then using (\ref{HC-series}) and  (\ref{eq:HCcombined}) we have,
\begin{align*}
    F_w(\l_0,\eta)(t) &= \der \big|_{\l=\l_0} (\l-\l_0)F_w(\l,\eta)(t)\\
    &=\der \big|_{\l=\l_0} \sum_{m=0}^\infty\eta_w \Big[e^{(\l-\rho)t}(\l-\l_0)p_R(\l)\Gamma_m(\l)\\
    &+ e^{(-\l -\rho)t}(\l-\l_0)p_R(\l)c(\l)\Gamma_m(-\l)\Big]e^{-mt}.
\end{align*}
If $\l_0 \in 1/2\mathbb{N}_0$, then the series expansion of $\Phi_{\l_0}$ and (hence) the series expansion of $F_w(\l_0,\eta)(t)$ is not valid. Thus, to get the series expansion we multiply by $\l-\l_0$ and take derivative at $\l_0$. To make it uniform, we use the above expression for all $\l_0$. Let \bes
a_m(\l) = (\l-\l_0)\Gamma_m(\l) \text{ and } b(\l) = (\l-\l_0) p_R(\l)c(\l).\ees
It follows from Lemma \ref{lemm: pole-zero-description} that $a_m$ and $b$ are holomorphic at $\lambda_0$. 
Then $F_w(\l_0, \eta)(t)$ can be written as
\bes
F_w(\l_0, \eta)(t)=p_0(t) + \sum_{m=1}^\infty p_m(t),
\ees
where \[p_0(t) = \eta_w[e^{(\l_0-\rho)t}p_R(\l_0) + e^{(-\l_0-\rho)t}b(\l_0)(-t)+ e^{(-\l_0-\rho)t}b'(\l_0)]\] and 

\begin{align*}
    p_m(t) &= \eta_w \Bigg[e^{(\l_0-\rho)t}a_m(\l_0)p_R'(\l_0) + e^{(\l_0-\rho)t}a_m'(\l_0)p_R(\l_0)\\
    &+ e^{(\l_0-\rho)t}a_m(\l_0)p_R(\l_0)t
     -e^{(-\l_0-\rho)t}b(\l_0)\Gamma_m(-\l_0)t \\
     &+ e^{(-\l_0-\rho)t}b'(\l_0)\Gamma_m(-\l_0) - e^{(-\l_0-\rho)t}b(\l_0)\Gamma_m'(-\l_0) \Bigg]e^{-mt}.
\end{align*}

First we will find the estimates of $a_m$, $a_m'$, $b$, $b'$, $\Gamma_m$ and $\Gamma_m'$. 
From Theorem \ref{thm:Gamma} we can choose an $R' > \rho$ and obtain the following estimate:
\[|\Gamma_m(-\l_0)| \leq M \frac{(1+|\l_0|)^{\text{deg} q_{R'}}}{|q_{R'}(-\l_0)|}(1+m)^{\chi}. \]
We note that $\Gamma_m$ has no pole on $\Re\lambda<0$.
Since $q_{R'}(\l)$ has no zeroes in $\Re \l < 0$ and $(1+|\l|)^{\text{deg} q_{R'}}/|q_R(\l)|$ is bounded uniformly in the region $\Re \l <0$, we have
\[|\Gamma_m(-\l_0)| \leq M (1+m)^{\chi}.\]
Therefore, for a fixed $t_0 >0$, we obtain 
\[|\Gamma_m(-\l_0)| \leq M_{t_0} e^{m t_0}.\]
Now, we shall calculate the estimate of $a_m$. 
Since $a_m$ is holomorphic at $\l_0$, integrating over a circle $\gamma_\epsilon$ centered $\l_0$ with radius  $\epsilon<1/2$ and using the estimate (\ref{est:gamma-1}) we get
\begin{align*}
    |a_m(\l_0)| &\leq M\oint_{\gamma_\epsilon} \frac{|\l-\l_0\|\Gamma_m(\l)|}{|\l-\l_0|}|d\gamma|\\
    & \leq \frac{M}{2} (1+m)^{\chi}.
\end{align*}

Now, we have to estimate $a_m'(\l_0)$. Since $a_m(\l)$ is holomorphic at $\l_0$, we obtain  from Cauchy integral formula that
\[a_m'(\l_0) = \frac{1}{2\pi i}\oint_{\gamma_\epsilon} \frac{a_m(\l)}{(\l-\l_0)^2} d\gamma.\]
Since $\Gamma_m(\l)$ is holomorphic on the closed curve $\gamma_\epsilon$, 
\begin{align*}
|a_m'(\l_0)| &\leq  \oint_{\gamma_\epsilon} \frac{|(\l-\l_0)\Gamma_m(\l)|}{|\l-\l_0|^2} |d\gamma|\\
& \leq M (1+m)^{\chi}  \oint_{\gamma_\epsilon} \frac{1}{|\l-\l_0|} |d\gamma| \\
& \leq M (1+m)^\chi.
\end{align*}
If $\l_0$ is not a pole of $c(\l)$ then $b(\l_0)=0$. Suppose that $\l_0$ is a pole of $c(\l)$ then $\l_0 = \mu_k= \rho-1-m_1^+-m_2^+ -2k >0$, for some $k \in \bN_0$. Using the property $\Gamma(z+1) = z\Gamma(z)$, we have
\begin{align*}
    \Gamma((-\rho + \lambda + 1+m_1^+ + m_2^+)/2) &= \frac{\Gamma((-\rho + \lambda + 1+m_1^+ + m_2^+)/2 +k+1)}{\prod_{i=0}^{k}((-\rho + \lambda + 1+m_1^+ + m_2^+)/2 +i)}\\
    &= 2^{k+1}\frac{\Gamma((\l-\mu_k)/2 +1)}{\prod_{i=0}^{k} (\l-\mu_i)}.
\end{align*}

 Thus, using $2^{1-2z}\sqrt{\pi}\Gamma(2z)= \Gamma(z)\Gamma(z+1/2)$ and $\Gamma(z + \alpha)/\Gamma(z+\beta) \sim z^{\alpha-\beta}$ ( \cite[1.18 (3)]{E81}), we obtain that for $d = {\text{deg}}\, p_R$ \begin{align*}
    |b(\l)| & \leq \left|(\l-\mu_k)p_R(\l) 2^{2\lambda} \frac{\Gamma((\rho+\lambda)/2)\Gamma(-\lambda)\Gamma((\lambda-\rho+1+m_1^+ + m_2^+)/2)}{\Gamma((\rho-\lambda)/2)\Gamma(\lambda)\Gamma((-\rho - \lambda + 1+m_1^+ + m_2^+)/2)}\right|\\
    &\leq \left| \frac{(\l-\mu_k)p_R(\l)\Gamma((\lambda-\mu_k)/2 +1)}{\prod_{i=0}^k(\l - \mu_i)}\frac{2^{2\lambda+k+1}\Gamma((\rho+\lambda)/2)\Gamma(-\lambda)}{\Gamma((\rho-\lambda)/2)\Gamma(\lambda)\Gamma((-\rho - \lambda + 1+m_1^+ + m_2^+)/2)}\right|\\
    &\leq \left|\frac{p_R(\l)\Gamma((\lambda-\mu_k)/2 +1)}{\prod_{i=0}^{k-1}(\l - \mu_i)}\frac{2^{2\lambda+k+1}\Gamma((\rho+\lambda)/2)\Gamma(-\lambda)}{\Gamma((\rho-\lambda)/2)\Gamma(\lambda)\Gamma((-\rho - \lambda + 1+m_1^+ + m_2^+)/2)}\right|\\
    &\leq M (1+|\l|)^{d}.
\end{align*}  
Thus, \[|b(\l_0)| \leq M (1+|\l_0|)^d,\]
and by Cauchy's integral formula we have
\[|b'(\l_0)|\leq M (1+|\l_0|)^{d}.\]

Now, we will estimate $\Gamma_m'(-\l_0)$. Since $\Gamma_m$ is holomorphic at $-\l_0$, using Theorem~\ref{thm:Gamma}, we get 
\begin{align*}
    |\Gamma_m'(-\l_0)| &\leq \oint_{\gamma}\frac{|\Gamma_m(-\l)|}{|\l+\l_0|^2}|d\gamma|\\
    &\leq M/2 (1+m)^\chi,
\end{align*}
where $\gamma$ is a circle centered at $-\l_0$ and radius $<\frac 12$.

Hence for a fixed $t_0>0$,
\bes |\Gamma_m(-\lambda_0)|, \Gamma'_m(-\lambda_0), |a_m(\lambda_0)|,  |a_m'(\lambda_0)|, |b(\lambda_0)|, |b'(\lambda_0)| \ees  all are dominated by $e^{m t_0}.$ Let $\text{deg} \,p_R=d$.
Then, we have
\begin{align*}
   &  \left|F_w(\lambda, \eta)(t) e^{(-\lambda_0 +\rho)t}-\eta_wp_R(\lambda_0)\right| \leq e^{-2\Re\lambda_0 t} |\eta_w b(\lambda_0)t| + e^{-2\Re\lambda_0 t} |\eta_w b'(\lambda_0)| \\ 
   &+ \sum M |\eta_w| \left[(1 +m)^{\chi+d} + t(1 +m)^{\chi+d} + e^{-2\Re\lambda_0 t} t(1 +m)^{\chi+d} + e^{-2\Re\lambda_0 t} (1 +m)^{\chi+d}\right] e^{-mt}\\
    &\leq e^{-2\Re\lambda_0 t} |\eta_w b(\lambda_0)t| + e^{-2\Re\lambda_0 t} |\eta_w b'(\lambda_0)| + M \,t\sum_{m=1}^\infty |\eta_w|e^{-m(t-t_0)}\\
     &\leq e^{-2\Re\lambda_0 t} |\eta_w b(\lambda_0)t| + e^{-2\Re\lambda_0 t} |\eta_w b'(\lambda_0)| + M\, |\eta_w|t e^{-(t-t_0)} \left(\Sigma_{m=0}^\infty e^{-m(t-t_0)}\right)\\
  &\leq e^{-2\Re\lambda_0 t} |\eta_w b(\lambda_0)t| + e^{-2\Re\lambda_0 t} |\eta_w b'(\lambda_0)| + M \,|\eta_w|t e^{-(t-t_0)} \frac{1}{1-e^{-(t-t_0)}}.  \end{align*}
This implies that, for $\Re\lambda_0>0$,
\bes
\left|F_w(\lambda, \eta)(t) e^{(-\lambda_0 +\rho)t}-\eta_w p_R(\lambda_0)\right|\rightarrow 0,
\ees
as $t\rightarrow \infty$. Therefore, if $\lambda_0$ is not a zero of $p_R$ (or, equivalently if it is not a pole of $E^\circ_w(\lambda, \eta)(t)$) then 
\bes
\left|F_w(\lambda_0, \eta)(t)\right| \asymp |p_R(\lambda_0)\eta_w| e^{(\Re\lambda_0-\rho)t},
\ees
as $t\rightarrow \infty$. Hence,  if $\Re\lambda_0>0$ and if $\lambda_0$ is not a pole of $E^\circ_w(\lambda, \eta)(t)$ then 
\be \label{asymp:eisen}
|\eE_w(\lambda_0, \eta)(t)|\asymp |\eta_w|e^{(\Re\lambda_0-\rho)t},
\ee
as $t\rightarrow \infty$. 

\noindent{\bf Subcase 1: $\Re\lambda_0> \rho$:}  We observe that $E^\circ_w(\lambda, \eta)(t)$ has no pole for $\Re\lambda> \rho$.
Therefore, it follows from (\ref{asymp:eisen}) that for $\Re\lambda_0> \rho$, $E^\circ_w(\lambda_0, \eta)(\cdot)$ is not bounded. \\

\noindent{\bf Subcase 2: $\Re\lambda_0=\rho$:}  We also observe that $E^\circ(\lambda, t)$ has no pole for $\Re\lambda=\rho$.
Therefore, it follows from (\ref{asymp:eisen}) that for $\Re\lambda_0=\rho$, $E^\circ(\lambda_0, \cdot)$ is a bounded function.\\

\noindent{\bf Subcase 3: $0 \leq \Re \l_0 <\rho$:}  By using the estimates of $|\Gamma_m(-\l_0)|, |\Gamma'_m(-\l_0)|,|a_m(\l_0)|, |a_m'(\l_0)|, |b(\l_0)|$ and $|b'(\l_0)|$ we obtain that
\[|p_m(t)| \leq M |\eta_w|e^{(\Re \l_0 -\rho)t} [(1+m)^\chi(1+t)(1+|\l_0|)^d]e^{-mt}.\] 
By using $(1+m)^\chi\leq (m+\chi)!/m!$, we obtain the following inequality
\begin{equation}\label{ineq: Gamma}
    \sum_{m=0}^\infty(1+m)^\chi e^{-mt} \leq \sum_{m=0}^\infty \frac{(m+\chi)!}{m!}e^{-mt} = \chi! (1-e^{-t})^{-\chi-1}.
\end{equation}
Using the above we obtain that
\[|F_w(\l_0, \eta)(t)| \leq M |\eta_w|(1+|\l_0|)^d (1+t)e^{(\Re \l_0 -\rho)t}.\]
Thus, in this case $F_w(\l_0,\eta)(t)$ is a bounded function in $t$. Note that, when $\Re \l_0 = 0$, then $\eE_w$ has no singularities. Thus, for $\l_0 = ix$ with $x \neq 0$, we have (from \ref{asymp:eisen})
\[|\eE_w(ix, \eta)(t)| \asymp |\eta_w| e^{-\rho t}.\]
\\
{\bf Case 2: $\l_0 \in -\bN_0$:} From Lemma~\ref{lemm: pole-zero-description},  we have that $\lambda_0=-k$ is a zero of $p_R\eE_w$. Thus,  $p_R(\l_0)E^\circ_w(\lambda_0, \eta)(t)=0$  for all $t>0$. \\

{\bf Case 3: $\Re\lambda_0<0$:}
In this case we assume that $\l_0$ is not a negative integer and $-R<\Re\lambda_0<0$. 
Then using (\ref{HC-series}) and  (\ref{eq:HCcombined}) we have,
\begin{equation}\label{eq: Re<00}
    \begin{split}
         F_w(\l_0,\eta)(t) &= \der \big|_{\l=\l_0} \eta_w \sum_{m=0}^\infty \Big[e^{(\l-\rho)t}(\l-\l_0)p_R(\l)\Gamma_m(\l)\\
    &+ e^{(-\l -\rho)t}(\l-\l_0)p_R(\l)c(\l)\Gamma_m(-\l)\Big]e^{-mt}.
    \end{split}
\end{equation}

Let \bes
g_m(\l) = (\l-\l_0)p_R(\l)\Gamma_m(\l),  h(\l) = p_R(\l)c(\l) \text{ and } d_m(\l)=(\l-\l_0)\Gamma_m(-\l).\ees
We observe that all these are holomorphic at $\l_0$. 

Again writing $F_w(\l_0,\eta)(t)$ as
\[F_w(\l_0,\eta)(t) = p_0(t) + \sum_{m=1}^\infty p_m(t),\]
with \[p_0(t) = \eta_we^{(\l_0-\rho)t}p_R(\l_0)+ \eta_we^{(-\l_0-\rho)t}h(\l_0)\] and
\[\begin{split}
    p_m(t) &= \eta_w [e^{(\l_0-\rho)t}g_m'(\l_0) + e^{(\l_0-\rho)t}tg_m(\l_0)
    + e^{(-\l_0-\rho)t}d_m(\l_0)h'(\l_0)\\
    &+e^{(-\l_0-\rho)t}d_m'(\l_0)h(\l_0)+ e^{(-\l_0-\rho)t}(-t)d_m(\l_0)h(\l_0)].
\end{split} \]

 Similar to case $\Re \l_0>0$, the functions $g_m(\l), h(\l)$ and $d_m(\l)$ can be dominated by $e^{m t_0}$ for some $t_0>0$. Then calculating similarly as above we get that
\bes
\left|F_w(\lambda_0, \eta)(t) e^{(\lambda_0 +\rho)t}- \eta_w h(\lambda_0)\right| \rightarrow 0,
\ees
as $t\rightarrow \infty$. 
We have $p_R(\l)\eE_w(\l,\eta)(t) \equiv 0$ if and only if $\l_0 \in -\mathbb{N}_{0}$ as the the polynomial $p_R$ encaptures the poles of $\eE$ described in Lemma~\ref{lemm: pole-zero-description}. It follows that when $\l_0 \notin -\mathbb{N}_{0}$ and $\Re \l_0 <0$ we have  $h(\l_0) \not= 0$. \\

\noindent{\bf Subcase 1: $-R\leq \Re\lambda_0<-\rho$ and $\l_0\notin -\bN_0$:}  In this case, we first observe (from the definition of $h$) that $h(\lambda_0)\not=0$. We have
\bes
\left|F_w (\lambda_0, \eta)(t)\right|\asymp |\eta_w|e^{-(\Re\lambda_0 +\rho)t},
\ees
as $t\rightarrow\infty$.  This implies $F_w(\lambda_0, \eta)(\cdot)$ is unbounded, in this case.

\vspace{.3cm}

\noindent{\bf Subcase 2: $\Re \l_0 = -\rho$: } Suppose $\rho$ is not a natural number, then $h$ is non zero at $\Re \l_0 = -\rho$ and 
\[|F_w(\l_0, \eta)(t)| \asymp  |\eta_w|. \]
Thus, $F_w(\l_0,\eta)(\cdot)$ is bounded. 

If $\rho$ is a natural number and $\Re \l_0 = -\rho$, $\l_0 \not= -\rho$, then also $h(\l_0)$ is non-zero. Therefore, $F_w(\l_0,\eta)(\cdot)$ is bounded in this case as well. \\

\noindent{\bf Subcase 3: $-\rho< \Re\lambda_0<0$ and $-R < \Re \l_0$:} It follows from (\ref{eq: Re<00}) that 
\bes
\left|F_w (\lambda_0, \eta)(t)\right| \leq C |h(\lambda_0)|(1+t) e^{-(\Re\lambda_0 +\rho)t},
\ees
as $t\rightarrow\infty$.
If $-\rho\leq \Re\lambda_0<0$, then $\Re\lambda_0 +\rho\geq 0$ and hence $F_w(\lambda_0, \eta)(t)$ is bounded.
\end{proof}

From the proof, we get the following improved estimate of normalized Eisenstein integrals (c.f. (\ref{est: normalized-Eisenstein})):
\begin{corollary}\label{cor:estimate} For $R>0$ and  a fixed $\eta \in \C^{\cW}$, there exists a constant $M>0$ such that for all $\lambda_0 \in -a^*(R)$ and $t>0$,  \be \label{eqn:estimate99}
  |p_R(\lambda_0)E^\circ_w(\lambda_0,\eta)(t)| \leq M \|\eta\|(1 + |\lambda_0|)^{\text{deg}\, p_R} (1+t) e^{(|\Re\lambda_0|-\rho)t}.
  \ee
  Moreover, for each fixed $\l_0$ we have the following:
\begin{enumerate}
    \item If $\l_0$ is a pole of $\eE_w$, then \bes
  |p_R(\lambda_0)E^\circ_w(\lambda_0,\eta)(t)| \leq M \|\eta\|(1 + |\lambda_0|)^{\text{deg}\, p_R} (1+t) e^{(|\Re\lambda_0|-\rho)t}.
  \ees
  \item If $\l_0$ is a non-positive integer then $p_R(\l_0)\eE_w(\l_0,\eta)(t) = 0$.
  \item If $\l_0$ is not a pole of $\eE$ and also not a non-positive integer, then
  \[|\eE_w(\l_0,\eta)(t)| \asymp M |\eta_w|e^{(|\Re\l_0|-\rho)t}.\]
\end{enumerate}

\end{corollary}
The above estimate (\ref{eqn:estimate99}) has been obtained  in \cite{A01} for $R = 1/4$ using Jacobi functions for real hyperbolic spaces $\mathrm{\mathrm{SO}}_e(p,q)/\mathrm{\mathrm{SO}}_e(p-1,q)$ with $q>1$ and in the general case we had only a rough estimate (\ref{est: normalized-Eisenstein}). 

\begin{corollary}\label{cor:derivativeest}
    For $R>0$ and  a fixed $\eta \in \C^{\cW}$, there exists a constant $M>0$ such that for all $\lambda \in -a^*(R)$ and $t>0$,  \bes
  | \left(\frac{\partial}{\partial \lambda}\right)^n p_R(\lambda)E^\circ_w(\lambda,\eta)(t)| \leq M \|\eta\| (1 + |\lambda|)^{\text{deg}\, p_R} (1+t)^{n+1} e^{(|\Re\lambda|-\rho)t}.
  \ees  
\end{corollary}
Proof of this follows from (\ref{eqn:estimate99}) and the Cauchy's integral formula by integrating over a circle with radius $\frac{1}{1+t}$.

\section{Hausdorff-Young inequality and Riemann-Lebesgue lemma}

 Our version of the Helgason-Johnson's theorem says that $p_R(\l)\eE_w(\l,\eta)(t)$ is bounded in $t$ for a fixed $\l \in -a^*(R) \cap S_1$. As a consequence of this theorem and Corollary~\ref{cor:estimate}, we have that for $f \in L^1(K \backslash G/H)$ the Fourier transform $\cF f$ extends meromorphically to $S_1$. Moreover, for a fixed $R > 0$, the function $p_R(-\l)\cF f(\l)$ is holomorphic on  $S_1^0 \cap a^*(R)$ and continuous on ${S_1} \cap a^*(R)$. 
 Now, we consider the region \[S_r = \{z \in \C{}: |\Re z| \leq (2/r -1)\rho\},\] for $1 \leq r \leq 2$. Observe that $J(t) \asymp e^{2 \rho t}$. Let $r'$ be such that $1/r' + 1/r =1$. Denote $C_c^\infty(\ia_{q})$ to be the space of all complex valued smooth functions with compact support on $\ia_{q}$. 

\begin{proposition}\label{prop:Eest}
	Let $1 \leq r \leq 2$. For $\l \in S_r^\circ \cap -a^*(R)$,  $w \in \cW$ and $\eta \in \C^{\cW}$, the function $t \mapsto p_R(\l)\eE_w(\l, \eta)(t)$ is in $L^{r'}((0,\infty), J(t)dt)$, where $1/r+1/r' = 1$.
\end{proposition}
\begin{proof} Using the estimate of corollary~\ref{cor:estimate} we have
    \begin{align*}
\left(\int_{0}^\infty |p_R(\l)\eE_w(\l,\eta)(t)|^{r'} J(t)dt\right)^{1/r'} &\leq M \|\eta\| \left(\int_{0}^\infty \left((1+|\l|)^{d} (1+t)e^{(|\Re \l|-\rho)t}\right)^{r'} e^{2\rho t}dt\right)^{1/r'}\\
& \leq M \|\eta\| (1+|\l|)^{d}\left(\int_{0}^\infty (1+t)^{r'}e^{r'(|\Re \l|-\rho)t + 2\rho t}dt\right)^{1/r'}.
\end{align*}
 
For the integral to converge we must have that the exponent $r'(|\Re \l|-\rho) + 2\rho <0 $. That is, 
$|\Re \l| < (1-2/r')\rho = (2/r-1)\rho$ which implies that if 
$ \l \in S_r^\circ $ then $p_R(\l)\eE_w(\l,\eta)(t)$ is in $L^{r'}$ for any $\eta \in \C^{\cW}$ and $w \in \cW$. 
\end{proof}
 
\begin{proposition}\label{prop:LrFT}
    For $1 \leq r \leq 2$ and $f \in L^r(K \backslash G/H)$, the Fourier transform exists on $S_r$ and extends meromorphically to the strip $S_r$. Furthermore, for a fixed $R > 0$, the function $p_R(-\l)\cF f(\l)\eta$ is holomorphic on  $S_r^0 \cap a^*(R)$ and has continuous extension on ${S_r} \cap a^*(R)$ for any $\eta \in \C^\cW$. 
\end{proposition}
\begin{proof}
    From Proposition~\ref{prop:Eest} and H\"older's inequality it follows that for $\l \in S_r \cap a^*(R)$
    \begin{align*}
        |p_R(-\l)\cF f(\l)\eta| & \leq \sum_{w\in \cW} \int_0^\infty |f(a_tw\cdot x_0) p_R(-\l)\eE_w(-\l,\eta)(t)|J(t)dt\\
        & \leq M \|f\|_{L^r(K \backslash G/H)} \|\eta\| (1+|\l|)^d.
    \end{align*}
    Thus, the $p_R(-\l)\cF f(\l)\eta$ exists and is continuous on the strip $S_r\cap a^*(R)$. Moreover, using Corollary~\ref{cor:derivativeest} we obtain
    \begin{align*}
        \left|\frac{\partial^m}{\partial \l^m}p_R(-\l)\cF f(\l)\eta\right| & \leq \sum_{w\in \cW} \int_0^\infty \left|f(a_tw\cdot x_0) \frac{\partial^m}{\partial \l^m}p_R(-\l)\eE_w(-\l,\eta)(t)\right|J(t)dt\\
        & \leq M \|f\|_{L^r(K \backslash G/H)}\|\eta\| (1+|\l|)^d \left(\int_0^\infty (1+t)^{(m+1)r'}e^{r'(|\Re \l|-\rho)t + 2\rho t}dt\right)^{1/r'},
    \end{align*}
    which converges if and only if $\l \in S_r^0 \cap a^*(R)$. Thus, $p_R(-\l)\cF f(\l)\eta$ is holomorphic in $S_r^0 \cap a^*(R)$.
\end{proof}

Then we have the following lemma, whose proof follows from the proof of Proposition~\ref{prop:LrFT}.

\begin{lemma}\label{lemma: H-YL1}
  Let $\l_0 \in S_1 \cap a^*(R)$ be a real number,  $w \in \cW$, $\eta \in \C^{\cW}$ and $d = \text{deg}\, p_R$. Then, we have 
  \[\sup_{\l \in \mathbb R}\left|\frac{p_R(-(\l_0 + i\l))\cF f(\l_0 + i\l)\eta}{(1+\l_0 + i\l)^d}\right| \leq M \|f\|_1 \|\eta\|, \quad {f \in L^1(K\backslash G/H)} ,\]
    for some $M > 0$ and for any $\eta \in \C{}^{\cW}$.
  
\end{lemma}

Now we prove an analogue of the Hausdorff-Young inequality.

\begin{theorem}\label{thm:HY}
	Let $1\leq r \leq 2$ and $\l_0 \in S_r \cap a^*(R)$ be real number.  Then there exists $M>0$ such that,
	\[\left( \int_{ \mathbb R} \left|\frac{p_R(-(\l_0 + i\l))\cF f(\l_0 + i\l)\eta}{(1+\l_0 + i\l)^d}\right|^{r'} d\l \right)^{1/r'} \leq M \|f\|_r \|\eta\|, \quad {f \in L^r(K\backslash G/H)}, \]
     for any $\eta \in \C{}^{\cW}$, with $d = \text{deg} \, p_R$.
\end{theorem}
\begin{proof}
	We will equip the space $i\mathbb R$ with Lebesgue measure $d\l$. Let $1\leq r \leq 2$ and fix an arbitrary $\l_0 \in S_r \cap a^*(R) $ such that $\l_0 \in \R{}$. Write $\mu = (2/r -1)^{-1}\l_0$. For every complex number $z$ such that $0 \leq \Re z \leq 1$ and $f \in C_c^\infty(K\backslash G/H)$, $\l \in i\mathbb R$,  we define a family of operators $T_{z}f(\l)$ by 
	\[T_{z}f(\l)(\eta) = \frac{p_R(-\l -z\mu)\cF f(\l + z\mu)\eta}{(1+ \l + z\mu)^d}, \qquad \eta \in \C^{\cW}.\]

 Fix $\eta \in \C^{\cW}$. Let $z = 1+ iy$. Then using Lemma~\ref{lemma: H-YL1} we get that 
    \begin{align*}
       \| {T}_{1+iy}f(\cdot)(\eta)\|_{L^\infty(i\R{})} & \leq  \sup_{\l \in \mathbb R}\left|\frac{p_R(-i\l -iy\mu-\mu)\cF{ f}(i\l + iy\mu + \mu)\eta}{(1+i\l +  iy\mu + \mu)^d}\right|\\
       & \leq C \sup_{\l \in \mathbb R} \left|\frac{p_R(-i\l -\mu)\cF{ f}(i\l + \mu)\eta}{(1+i\l + \mu)^d}\right|\\
       & \leq M  \|f\|_{L^1} \|\eta\|.
    \end{align*}
 Let $z = iy$. Then by (\ref{Plncherel}) and using the fact that on $i\R{}$ there are no poles of $\cF$, we have
\begin{align*}
    \|{T}_{iy}f(\cdot)\eta\|_{L^2(i\mathbb R)}^2 &= \int_{\mathbb R} \left|\frac{p_R(-i\l -iy\mu)\cF{ f}(i\l + iy\mu)\eta}{(1+i\l +  iy\mu)^d}\right|^2 d\l\\
    & \leq C \int_{\mathbb R} |\cF{ f}(i\l + iy\mu)\eta|^2 d\l\\
    & \leq C \int_{\mathbb R} |\cF{ f}(i\l)\eta|^2 d\l \\
    & \leq C \|f\|^2_{L^2(X)}\|\eta\| . 
\end{align*}

Now, consider the function 
\[\phi(z) = \int_{\R{}} {T}_{z}f(i\l)(\eta)\overline{g(i\l)} d\l, \quad f \in C_c^\infty(K \backslash G/H), \quad g \in C_c^\infty(i\R{}). \]
By definition of ${T}_{z}f$, the function $\phi$ is holomorphic on $0 < \Re z <1$ and continuous on $0 \leq \Re z \leq 1$. Using Lemma~\ref{lemma: H-YL1} we derive that $\|{T}_{z}f(\cdot)\eta\|_{L^\infty(i\R{})} \leq M \|f\|_{L^1(X)} \|\eta\|$
and
\[\log |\phi (z)| \leq M \log (\|f\|_{L^1(X)} \|\eta\| \|g\|_{L^1(i\R{})}).\]
Hence, $\{T_z\}$ is an admissible family of operators. Thus, by complex interpolation theorem  we have for $\theta_r = (2/r -1)$,
\[\|{T}_{\theta_r} f(\cdot)\eta \|_{L^{r'}(i\R{})} \leq M \|f\|_{L^r(X)}\|\eta\| \quad \forall f \in L^{r}(K \backslash X).\]
That is, for $\l_0 \in \R{}$ and $ \l_0 \in a^*(R)$,
\[ \left( \int_{\R{}} \left|\frac{p_R(-(\l_0 + i\l))\cF{ f}(\l_0 + i\l)(\eta)}{(1+\l_0 + \l)^d}\right|^{r'} d\l \right)^{1/r'} \leq M \|f\|_r \|\eta\|, \quad |\l_0| \leq (2/r -1)\rho.\]

We finally obtain that for real number $\l_0$ with $\l_0 \in S_r \cap a^*(R)$ ,  $w \in \cW$, $\eta \in \C^{\cW}$, we have

\[ \left( \int_{\R{}} \left|\frac{p_R(-(\l_0 + i\l))\cF f(\l_0 + i\l)\eta}{(1+\l_0 + \l)^d}\right|^{r'} d\l \right)^{1/r'} \leq M \|f\|_r \|\eta\|.\]

\end{proof}

\begin{Remark}
    If $-\l_0\in  S_r \cap -a^*(R)$ is not a singularity of $\eE$, that is $p_R(-\l_0) \neq 0$, then $p_R(-\l_0 - i\l)/(1+\l_0 + i\l)^d \asymp \, const$. Therefore, we can rewrite the statement as
    \[\left( \int_{\R{}} \left|\cF f(\l_0 + i\l)\eta\right|^{r'} d\l \right)^{1/r'} \leq M \|f\|_r \|\eta\|.\]
\end{Remark}
\begin{remark}
 Fix $\eta \in \C^{\cW}$. Let $1\leq r<2$, $f \in L^r(K\backslash G/H)$. Then $p_R(-\l)\cF f(\l)(\eta)$ exists almost everywhere on the lines $|\Re \l| = (2/r-1)\rho$ (cf. Proposition~\ref{prop:Eest}).
\end{remark}

We now prove an analogue of Riemann-Lebesgue Lemma.
\begin{theorem}
    Let $1 \leq r <2$ and $\l_0 \in S_r^\circ \cap a^*(R)$ be a real number. Then for any $f \in L^r(K \backslash G/H)$ and for each fixed $\eta$,
    \[\lim_{\l\in\mathbb R, |\l| \rightarrow \infty} |\cF f(\l_0 + i\l)\eta| = 0.\]
\end{theorem}
\begin{proof}
    As a consequence of the proof of Proposition~\ref{prop:Eest} we have that for $\l_0 + i\l \in S_r^\circ \cap a^*(R)$ \[|p_R(-\l_0-i\l)\cF f(\l_0+i\l)\eta| \leq M' \|\eta\|(1+|\l_0+i\l|)^{\text{deg} p_R}\|f\|_r.\]
For $\epsilon>0$, get a $\phi \in C_c^\infty(K \backslash G/H)$ such that 
$\|f-\phi\|_r < \epsilon$. From Proposition~\ref{prop:paleyweiner} we obtain that for large enough $|i\l|$
\[|p(-\l_0-i\l)\cF\phi(\l_0+i\l)\eta| \leq M(1+|\l_0+i\l|)^{-n}\|\eta\| < \epsilon.\]
for some $\epsilon >0$. Thus, 
\begin{align*}
    &|p_R(-\l_0-i\l)\cF f(\l_0+i\l)\eta|\\ &\leq |p_R(-\l_0-i\l)(\cF f(\l_0+i\l)\eta - \cF \phi(\l_0+i\l)\eta)| 
    + |p_R(-\l_0-i\l)\cF \phi(\l_0+i\l)\eta|\\
    & \leq \|\eta\|(1+|\l_0+i\l|)^{\text{deg} p}\epsilon + \epsilon
\end{align*}
For large enough $|\l|$, $p_R(-\l_0-i\l)$ has no zeroes and thus, we obtain that \[\lim_{|\l|>0, |\l| \rightarrow \infty} |\cF f(\l_0 + i\l)\eta| = 0.\]
\end{proof}
\begin{remark}
If $r = 1$, we can include the boundary point $|\l_0| = \rho$ in the theorem above.   
\end{remark} 

\section{Fixed K-type}

In this section we will study fixed left $K$-type normalized Eisenstein integrals on pseudo- Riemannian real hyperbolic spaces $\mathrm{\mathrm{SO}}_e(p,q)/\mathrm{\mathrm{SO}}_e(p-1, q)$. Our aim is to characterise bounded fixed left $K$-type normalized Eisenstein integrals and then we prove Hausdorff-Young inequality in this case. To prove these results (among many thing) we first need the explicit formula of the $c$-function. Due to inability in obtaining this in general split rank one/rank one cases, we restrict ourselves to the pseudo-Riemannian real hyperbolic spaces. First, we give required preliminaries and then give an outline of the proof of the main theorems in this section. 

\vspace{.4cm}

\subsection{\bf Geometry and Structure of pseudo-Riemannian real hyperbolic spaces:}  
In this section we give the required preliminaries of the pseudo-Riemannian real hyperbolic spaces. Let $e_1 = (1,0,...,0)\in \mathbb R^{p+q}$. Then $H= \mathrm{\mathrm{SO}}_e(p-1,q)$ is identified as subgroup of $G$ which is the stabilizer group of $e_1$. Let $K = \mathrm{SO}(p) \times \mathrm{SO}(q)$, $B = S^{p-1} \times S^{q-1}$ and 
\[\rho = \frac{p+q -2}{2}.\] 
Let us consider the following spherical co-ordinates on $X$:
\[x(t,b) = \cosh t u + \sinh t v, \quad b = (u,v) \in B.\]
We have $\fg_+\cong \mathfrak{so}(p-1)\times \mathfrak{so}(1, q)$. For $1\leq i, j\leq p+q$, let $E_{ij}$ be the $(p+q)\times (p+q)$ matrix with $1$ on the $(i, j)$th entry and zero on all other entries. Let \bes
Y=E_{p+q, 1} + E_{1, p+q}.
\ees
Then $\aq=\R{} Y\cong \R{}$. Let $\Sigma(\aq, \fg)=\{\pm\alpha\}$, where $\alpha(Y)=1$ and the Weyl group is $W\cong \{\pm 1\}$. When $q>1$, $\Sigma(\aq, \fg_+)=\Sigma(\aq, \fg)$ and  $W_{K\cap H}\cong \{\pm 1\}$. Therefore, $\mathcal W=\{1\}$ for $q>1$.
But if $q=1$, $\Sigma(\aq, \fg_+)=\emptyset$  and $W_{K\cap H}\cong \{1\}$. Thus, if $q=1$ we have that $\mathcal W=\{\pm 1\}$. 

The $\sigma$-minimal parabolic subgroup $P$ of $G=\mathrm{SO}_{e}(p,q)$ is the stabilizer of the line $\R{}(1,0,...,0,1)$. Moreover,  $P=MAN$ where

\[A = \mr{exp}(\aq) = \left  \{a_t =\begin{pmatrix}
    \cosh{t} &0&\sinh{t}\\
    0&Id&0\\
    \sinh{t}&0&\cosh{t}
\end{pmatrix}: t\in \R{} \right\},\]

\[N = \mr{exp}(\fn) =\left\{ \begin{pmatrix}
    1+\frac{1}{2}\|v\|^2&v^t&-\frac{1}{2}\|v\|^2\\
    v&Id&-v\\
    \frac{1}{2}\|v\|^2&v^t&1-\frac{1}{2}\|v\|^2
\end{pmatrix}: v\in \R{p+q-2}\right\}\]
and 
\[M = \left\{\begin{pmatrix}
    \epsilon & 0 &0\\
    0& m & 0\\
    0 & 0& \epsilon
\end{pmatrix}: \epsilon = \pm 1, m \in \mathrm{SO}(p-1,q-1) \right\},\] for $q>1$, 
and  $\epsilon=1$ in the above description of $M$ for $q=1$.  Thus, there is only one $P$ open orbit on $X$ when $q>1$, which is $$\cO = \{x \in X: -x_1 + x_{p+q} \neq 0\}.$$ For $q=1$, there are two open $P$-orbits  on $X$ which are
\[\cO_1 := P\cdot e_n = \{x \in X : x_1 - x_{n+1} >0\},\] and \[\cO_2 = P \cdot -e_n =\{x \in X : x_1 - x_{n+1} <0\}.\] 
Also, we have $m_1^+=m_\alpha^+=q-1$, $m_1^-=m_\alpha^-=p-1, m_2^+=0$ and  $m_2^-=0$. Then the Jacobian $J$ (defined in (\ref{jacobian})) associated with the decomposition (\ref{decomposition}) becomes \begin{equation}\label{eq:Jacobian}
  J(t)=\cosh^{p-1}t \sinh^{q-1}t.  
\end{equation} 
Let $\Delta$ be the  Laplace-Beltrami operator on $X$. Then the radial part of the Laplace Beltrami operator is given by
\[L(\Delta) = \frac{1}{J(t)}\frac{d}{dt} (J(t) \frac{d}{dt}).\]

\vspace{.5cm}

We will now consider the two cases when $q>1$ and $q=1$ separately, as their structure and analysis are different. 

\vspace{.5cm}

\noindent{\bf Case when $p\geq1$ and $q>1$:} 
In this case, for $\eta \in \C{}$, we consider the function 
\[j(\l, \eta)(x) = \begin{cases}
    |x_1 - x_{p+q}|^{\l-\rho}\eta & if \;x \in \cO\\
    0 & if \; x \notin \cO.
\end{cases}\]
Observe that we can write $j(\l, \eta)(x) = \eta j(\l, 1)(x)$. 
The Eisenstein integral is defined as 
\[E (\lambda,\eta)(x) = \int_K j(\l, \eta)(k^{-1}x) dk.\]

By \cite[Example 6.5]{HS94} and \cite[p. 162]{HS94}, we have that $E(\lambda, \eta)(a_t\cdot e_1)$ may have singularities at $\lambda \in \ia_{q}$ when $p > q+2$. The normalized Eisenstein integral $\eE$ becomes (from (\ref{eq:eisen}) and (\ref{eq: Re<0})),
\be \label{exp:normalized-eisenstein} E^0(\lambda, \eta)(a_t\cdot e_1) := \eta \,2^{\lambda-\rho}\frac{\Gamma(\frac{\lambda+\rho}{2})\Gamma(\frac{\lambda-\rho+q}{2})}{\Gamma(\lambda)\Gamma(\frac{q}{2})}{}_2F_1\Big(\frac{\lambda+\rho}{2}, \frac{-\lambda +\rho}{2};\frac{q}{2}; -\sinh^2{t}\Big).
\ee
 This does not have any singularities on $i\mathbb R$.
 In this case we have the following expression of $c$-function from Lemma \ref{lem:c-function}: \bes
c(\l)= 2^{2\lambda}\frac{\Gamma((\rho+\lambda)/2)\Gamma(-\lambda)\Gamma((\lambda-\rho+q)/2)}{\Gamma((\rho-\lambda)/2)\Gamma(\lambda)\Gamma((-\lambda-\rho+q)/2)}.
 \ees

\noindent{\bf Case when $p> 1, q=1$:}  In this case we consider the semisimple symmetric space ${X}= \mathrm{\mathrm{SO}}_e(p,1)/\mathrm{\mathrm{SO}}_e(p-1,1)$ with $p\geq 1 1$.
Let $H= \mathrm{\mathrm{SO}}_e(p-1,1)$ and $K = \mathrm{SO}(p)\times \mathrm{SO}(1)$ and 
\[\rho = \frac{p -1}{2}.\]  We have $m_1^+=m_\alpha^+=0$ and $m_1^-=m_\alpha^-=p-1, m_2^+=0, m_2^-=0$. Then the Jacobian $J$ (defined in (\ref{jacobian})) associated with the decomposition (\ref{decomposition}) becomes \bes J(t)=\cosh^{p-1}t.\ees
Let us define the function  
\[j(\l, \eta)(x) = \begin{cases}
    |x_1- x_{p+1}|^{\l-\rho}\eta_i & \text{if}\; x \in \cO_i\\
    0 & \text{if}\; x \notin \cO_1 \cup \cO_2.
\end{cases}\]
for $\eta \in \C^{2}$.
The Eisenstein integral is then defined as
\[E(\l,\eta)(x) = \int_K j (\l,\eta)(k^{-1}x) dk.\]

The normalized Eisenstein integral  $\eE(\l,\eta) (x)$ is a regular function on $X$ for every $\eta \in \C^{2}$. Let $e_{1}= (1,...,0,0)$ and consider {the disjoint orbits} 
\[\{a_t \cdot e_{1} = \cosh t e_1 + \sinh t e_{p+1}:t>0\} \quad \text{and} \quad\{ a_t \cdot -e_1 =  -\cosh t e_1 - \sinh t e_{p+1}:t>0\}.\] We denote \[\eE_w(\l, \eta)(t) := \eE(\l,\eta)(a_tw \cdot e_1).\] Thus, for $w = 1$,  $\eE_1(\l, \eta)(t) = \eE(\l,\eta)(a_t \cdot e_1)$ and for $w=-1$, we have 
 $\eE_{-1}(\l,\eta)(t) = \eE(\l,\eta)(a_t \cdot -e_1)$.
The normalized Eisenstein integral $\eE$ becomes (from (\ref{eq:eisen}) and (\ref{eq: Re<0})),
\begin{equation}\label{eq:eisen12}
    \eE_w(\l,\eta)(t) =   \eta_w\frac{2^{\l-\rho}\Gamma((\l+\rho)/2)\Gamma((\l-\rho+1)/2)}{\Gamma(\l)\Gamma(1/2)}{}_2F_1\left(\frac{\rho+\l}{2}, \frac{\rho-\l}{2}; \frac{1}{2}; -\sinh^2 t\right).
\end{equation}
  Also in this case, we have the following expression of $c$ function from Lemma \ref{lem:c-function}:
\bes
c(\lambda)=2^{2\lambda}\frac{\Gamma((\rho+\lambda)/2)\Gamma(-\lambda)\Gamma((\lambda-\rho+1)/2)}{\Gamma((\rho-\lambda)/2)\Gamma(\lambda)\Gamma((-\lambda-\rho+1)/2)}.
\ees
\begin{Remark}
   We omit the case when $p=q=1$ because in this case the group $G=\mathrm{SO}_{e}(1, 1)$ is abelian and hence the general theory is not valid for this group.
\end{Remark}

Let $\delta_k$ be an irreducible representations of $\mathrm{SO}(p)$ on the space of spherical harmonics $\mathcal{H}_k$ of degree $|k|$ . If $p=1$, then $k=0$, for $p=2$, $k \in \mathbb{Z}$, and for $p>2$, $k \in \mathbb{N}_{0}$. Let $\delta_l$ be the irreducible representations of $\mathrm{SO}(q)$ and write $\delta_{k,l} = \delta_k \otimes \delta_l$ . Then $\delta_{k,l}$ are the irreducible representations of $K=\mathrm{SO}(p)\times \mathrm{SO}(q)$ on $\mathcal{H}_k \otimes \mathcal{H}_l$. Let $\chi_{k,l}$ be the character of the representation $\delta_{k,l}$.  We say that a function $f \in C_c^\infty(X)$ is of left $(k,l)$-type if \[ \chi_{k,l}*_K f  = f, \]
that is,
$$\int_{K} \chi_{k,l} (k) f(k^{-1}x)  dk = f(x), \qquad x \in G.$$
We denote the space of all left $(k,l)$-type smooth functions with compact support and $(k,l)$-type integrable functions by $C_c^\infty (X)^{k,l}$ and $L^1(X)^{k,l}$, respectively.

The $(k,l)$-th component of the Laplace-Beltrami operator  on $X$ is given by:

\begin{align*}
    \Delta^{k,l}f(t) =& \frac{1}{\sqrt{J(t)}}\left[\frac{\partial^2}{\partial t^2}((\sqrt{J(t)} f(t)) -  \frac{\partial^2}{\partial t^2}(\sqrt{J(t)})f(t)\right] \\
&+ k(k+p-2) \frac{1}{\cosh ^{2}t} f(t) + l(l+q-2) \frac{1}{\sinh ^{2}t} f(t).
\end{align*}

We  consider the following differential equation 
\[\Delta^{k,l} f = (\lambda^2 - \rho^2) f.\]

Then there exists a unique formal series (see \cite[p. 78]{A01} for $q>1$ and \cite[p.66]{VS97} for general case) given by 
\[\Phi_\lambda^{k,l} (t)= e^{(\lambda-\rho)t}\sum_{m=0}^\infty {\Gamma}^{k,l}_m(\lambda)e^{-mt},\]
for $t>0$, that satisfies the above differential equation with $\Gamma_m^{k,l}(\l) \in \C{}$ and $\Gamma^{k,l}_0(\l) =1$.

The normalized Eisenstein integral $\eE_{k,l}$ is an eigenfunction of the differential operator $\Delta^{k,l}$ that satisfies
\[\Delta^{k,l} \eE_{k,l}(\l,\eta)(t) = (\l^2-\rho^2)\eE_{k,l}(\l,\eta)(t).\]

For $w \in \cW$, we write \[\eE_{k,l,w}(\lambda,\eta)(t): = \eE_{k,l}(\l,\eta)(a_t w\cdot e_{1}).\] We then have $\eE_{k,l,w}$ satisfies \be \label{laplace-kl}\Delta^{k,l} \eE_{k,l,w}(\l,\eta) = (\lambda^2-\rho^2)\eE_{k,l,w}(\l,\eta),\ee
 for any $w \in \cW$ and for almost all $\Re\l>0$,
 \[\lim_{t \rightarrow \infty}e^{(\rho-\l)t}\eE_{k,l,w}(\l,\eta)(t) = \eta_w.\] Furthermore, there exists an endomorphism  $C^{k,l}(\l) : \C^\cW \rightarrow \C^{\cW}$ that is meromorphically dependent on $\l$ such that
 \[\eE_{k,l,w}(\l,\eta)(t) = \Phi^{k,l}_\l(t)\eta_w + \Phi^{k,l}_{-\l}(t) [C^{k,l}(\l)\eta]_w, \quad \l \notin 1/2\mathbb{N}_{0}.\]
 We also have the following relation 
 \[\eE_{k,l,w}(-\l,\eta) = \eE_{k,l,w}(\l,C^{k,l}(\l)\eta).\]
 The above differential equation (\ref{laplace-kl}) can be transformed to a hypergeometric differential equation by substituting $z = -\sinh ^2t$. Thus, from the above asymptotics we obtain that
  \begin{align*}
     \eE_{k,l,w}(\l,\eta)(t) =& \eta_w2^{\l-\rho-|k|-|l|}\cosh^{|k|}t \sinh^{|l|}t\,\frac{\Gamma((\l+\rho+|k|+|l|)/2)\Gamma((\l-\rho-|k|+|l|)/2)}{\Gamma(\l)\Gamma(q/2 + |l|)}\\
     &\times {}_2F_1\left(\frac{\l+\rho+|k|+|l|}{2}, \frac{-\l+\rho+|k|+|l|}{2}; \frac{q}{2}+|l|; -\sinh ^2t\right).
 \end{align*}

By similar steps as in the proof of Lemma \ref{lem:c-function}  we get for $q>1$,
\[C^{k,l}(\l) = 2^{2\l} \frac{\Gamma(1/2(\l+\rho+|k|+|l|))\Gamma(1/2(\l-\rho+q-|k|+|l|))\Gamma(-\l)}{\Gamma(1/2(-\l+\rho+|k|+|l|))\Gamma(1/2(-\l-\rho+q-|k|+|l|))\Gamma(\l)},\]
and for $q=1$, $C^{k,l}(\l) = [c_{ij}(\l)]$ where

\begin{align*}
    c_{11}=c_{22} &= 2^{2\l} \frac{\Gamma(1/2(\l+\rho+|k|+|l|))\Gamma(1/2(\l-\rho+q-|k|+|l|))\Gamma(-\l)}{\Gamma(1/2(-\l+\rho+|k|+|l|))\Gamma(1/2(-\l-\rho+|k|+|l|))\Gamma(\l)},\\
    c_{12} = c_{21} &=0.
\end{align*}

Since $C^{k,l}$ is a diagonal matrix for $q=1$, by combining the case $q>1$ and $q=1$ we can write:
\[\eE_{k,l,w}(\l,\eta)(t) = \Phi^{k,l}_{\l}(t) \eta_w + c^{k,l}(\l)\eta_w\Phi_{-\l}^{k,l}(t),\]
where
\[c^{k,l}(\l) = 2^{2\l} \frac{\Gamma(1/2(\l+\rho+|k|+|l|))\Gamma(1/2(\l-\rho+q-|k|+|l|))\Gamma(-\l)}{\Gamma(1/2(-\l+\rho+|k|+|l|))\Gamma(1/2(-\l-\rho+|k|+|l|))\Gamma(\l)}.\]

 The function $\eE_{k,l,w}(\l,\eta)$ is of $(k,l)$-type.
\begin{definition}
    For $f \in C_c^\infty (X)^{k,l}$ we define the Fourier transform as 
    \begin{align*}
        \cF^{k,l}f(\l)(\eta) &= \int_X f(x)\eE_{k,l}(-\l,\eta)(x)dx\\
        &=\sum_{w \in \cW} \int_0^\infty f(a_tw.e_1) \eE_{k,l}(-\l,\eta)(a_tw\cdot e_1) J(t)dt,
    \end{align*}
    where $\l \in \fa^*_{q, \C}$.
\end{definition}
Fix an $R > 0$. Then there exists a polynomial $p_R^{k,l}$ such that $p_R^{k,l}(\l)\eE_{k,l,w}(\l,\eta)$ is holomorphic in $-a^*(R)$. Let \bes 
q_R(\l) = (\l-1/2)(\l-1)...(\l-\left \lfloor R/2 \right \rfloor).
\ees
Again from \cite[Thm 7.4]{VS97}, we have the following estimates:
\[|q_R(\l)\Gamma_m^{k,l}(\l)| \leq M (1+m)^\chi (1+|\l|)^{\text{deg}\, q_R},  \]
for some $\chi >0$,  $\l \in \overline{a^*(R)}$ and for all $m \in \mathbb{N}_{ 0}$.

 Let \bes \Psi_\lambda(t) = \sqrt{J(t)}\Phi_\lambda^{k,l}(t) = e^{\lambda t}\sum_{m=0}^{\infty} \ti{\Gamma}^{k,l}_m(\lambda)e^{-mt}.\ees 
Then $\Psi_\lambda$ satisfies that
\[\Psi_\lambda''(t) + (k(k+p-2)c(t)+l(l+q-2)s(t) - d(t))\Psi_\lambda(t) = (\lambda^2-\rho^2)\Psi_\lambda(t),\]
where \bes c(t) = \cosh ^{-2}t = \sum_{m=0}^\infty c_m e^{-mt},\ees 
\bes
s(t) = \sinh ^{-2}t = \sum_{m=0}^\infty s_m e^{-mt}, \ees and \bes
d(t) = J''/\sqrt{J} = \sum_{m=0}^\infty d_m e^{-mt}.\ees The constants are 
\begin{align*}
    d_0 = \rho^2, &\quad d_m = ((q-1)(q-3)+(-1)^m(p-1)(p-3))m,\\
    c_0 = 0, &\quad c_{2m} = 4(-1)^m m, \quad c_{2m+1} = 0,\\
    s_0 = 0, & \quad s_{2m} = 4m, \quad s_{2m+1} = 0.
\end{align*}

By substituting we obtain the following recursion formula:
\bes
    (\lambda^2 -\rho^2)\ti{\Gamma}_m^{k,l}=(\lambda-k)^2\ti{\Gamma}_m^{k,l} + k(k+p-2)\sum_{i=0}^m \ti{\Gamma}_i^{k,l} c_{m-i} + l(l+q-2) \sum_{i=0}^m \ti{\Gamma}_i^{k,l} s_{m-i} - \sum_{i=0}^m \ti{\Gamma}_i^{k,l} d_{m-i}.
    \ees
    That is,
    \bes
m(2\l-m)\ti{\Gamma}_m^{k,l}=\rho^2\ti{\Gamma}_m^{k,l} + k(k+p-2)\sum_{i=0}^m \ti{\Gamma}_i^{k,l} c_{m-i} + l(l+q-2) \sum_{i=0}^m \ti{\Gamma}_i^{k,l} s_{m-i} - \sum_{i=0}^m \ti{\Gamma}_i^{k,l} d_{m-i},
\ees
which implies that
\[  m(2\l-m)\ti{\Gamma}_m^{k,l}= k(k+p-2)\sum_{i=0}^{m-1} \ti{\Gamma}_i^{k,l} c_{m-i} + l(l+q-2) \sum_{i=0}^{m-1} \ti{\Gamma}_i^{k,l} s_{m-i} - \sum_{i=0}^{m-1} \ti{\Gamma}_i^{k,l} d_{m-i}.\]

Therefore, the poles of $\ti{\Gamma}_m^{k,l}$ lies in the set $\{1/2,1,..., m/2\}$. With 

\[\Gamma_m^{k,l}(\lambda) = \sum_{i=0}^{m}b_i\ti{\Gamma}_{m-i}^{k,l}(\lambda),\]

we obtain that the poles of $\Gamma^{k,l}_m \subseteq \{1/2,1,..., m/2\}$. 

{\begin{lemma}\label{lemma:singktype} The following holds true:
\begin{enumerate}
    \item The singularities of $\Gamma_m^{k,l}$ are at most simple poles in $\{1/2,1,...,m/2\}$.
    \item The singularities of $\Phi^{k,l}_{\l}$ lies in $1/2\mathbb{N}_{0}$.
    \item The poles of $c^{k,l}(\l)$ are $\{-\rho -|l|-|m|-2k_1:k_1 \in \mathbb{N}_{0}\} \cup \{\rho -q +|l|-|m| - 2k_2:k_2 \in \mathbb{N}_{0}\}\cup \{k_3 \in \mathbb{N}\}$.
    \item The poles of $\eE_{k,l,w}(\cdot,\eta)(t)$ are equal to $\{-\rho  -|l|-|m|-2k_1:k_1 \in \mathbb{N}_{0}\} \cup \{\rho -q +|l|-|m| - 2k_2:k_2 \in \mathbb{N}_{0}\}$.
    \item The zeros of $c^{k,l}(\l)$ are $\{\rho +|l|+|m|+2k_1:k_1 \in \mathbb{N}_{0}\} \cup \{-\rho +q -|l|+|m| + 2k_2:k_2 \in \mathbb{N}_{0}\}\cup \{-k_3 \in \mathbb{N}\}$
    \item  The zeros of $\eE_{k,l,w}(\cdot,\eta)(t)$  are $\{-k: k \in \mathbb{N}_{0}\}$.
\end{enumerate}
    
\end{lemma}}

Then we have the following analogue of Helgason-Johnson theorem for left $(k, l)$ type normalized Eisenstein integrals.
\begin{theorem}[Helgason-Johnson]\label{thm:ktypehj}
    Fix $R>0$, $w \in \cW$ and $\eta \in \C^{\cW}$.
   \begin{enumerate}
       \item Suppose $\rho$ is not an integer. 
       \begin{enumerate}
           \item If $\l_0 \in (S_1 \cap -a^*(R)) \cup -\bN_0$, then the function $t \rightarrow p_R^{k,l}(\l)\eE_{k,l,w}(\l,\eta)(t)$ is bounded on $(0,\infty)$.
           \item Conversely, for $\l_0 \in -a^*(R)$ and for non-zero $\eta_w$, if the function $t\mapsto p_R^{k,l}(\l_0)\eE_{k,l,w}(\l_0,\eta)(t)$ is bounded on $(0,\infty)$, then $\l_0 \in S_1  \cup -\mathbb N_0$. 
       \end{enumerate}
   \item Suppose $\rho$ is an integer. Let $A = \{\l \in \mathbb{Z} : \rho < \l \leq \rho -q +|l|-|m|\}$ and let $\l_0 \notin A$.
   \begin{enumerate}
       \item If $\l_0 \in (S_1 \cap -a^*(R)) \cup -\bN_0$, then the function $t \rightarrow p_R^{k,l}(\l)\eE_{k,l,w}(\l,\eta)(t)$ is bounded on $(0,\infty)$.
       \item Conversely, for $\l_0 \in -a^*(R)$ and for non-zero $\eta_w$, if the function $t\mapsto p_R^{k,l}(\l_0)\eE_{k,l,w}(\l_0,\eta)(t)$ is bounded on $(0,\infty)$, then $\l_0 \in S_1  \cup -\mathbb N_0$.
   \end{enumerate}
     
   \end{enumerate}
\end{theorem}

\begin{remark}
The set $A$ is non-empty if and only if $q < |l|-|m|$. For the case when $\l_0 \in A$ and $\rho$ is an integer, the function $\eE_{k,l,w}(\l,\eta)$ has a simple pole at $\l_0$ but $c^{k,l}(\l)$ has a double pole at $\l_0$.  Thus, the earlier techniques do not work in this case. 
\end{remark}

\begin{proof}[Proof of Theorem~\ref{thm:ktypehj}]
Proof of this theorem is similar to the $K$-invariant case. Thus, we give only an outline of the proof. 
Fix an $R > 0$. Let $p_R^{k,l}(\l)$ be the polynomial such that the function $F^{k,l}_w(\l,\eta) := p_R^{k,l}(\l)\eE_{k,l,w}(\l,\eta)(t)$ is holomorphic in $-\fa_q^*(R)$. Then
\begin{align*}
    F^{k,l}_w(\l_0,\eta)(t) &= \der \big|_{\l=\l_0} \eta_w \sum_{m=0}^\infty \Big[e^{(\l-\rho)t}(\l-\l_0)p_R^{k,l}(\l)\Gamma^{k,l}_m(\l)\\
    &+ e^{(-\l -\rho)t}(\l-\l_0)p_R^{k,l}(\l)c^{k,l}(\l)\Gamma_m^{k,l}(-\l)\Big]e^{-mt}.
\end{align*}

For $\l_0 \in -a^*(R) $, by following similar steps and estimates as in the case of $K$-invariant we obtain that 
\begin{equation}\label{eq:a}
    \begin{split}
        |F^{k,l}_w(\l_0,\eta)e^{(-\l_0+\rho)t}-\eta_wp_R^{k,l}(\l_0)|& \rightarrow 0 \qquad \Re \l_0 >0,\\
|F^{k,l}_w(\l_0,\eta)e^{(\l_0+\rho)t}-\eta_wp_R^{k,l}(\l_0)c^{k,l}(\l_0)|& \rightarrow 0 \qquad \Re \l_0 <0.
    \end{split}
\end{equation}

Moreover, we obtain the following estimate analogous to the case of $K$-invariant:
\begin{equation}\label{eq:k,l}
    |F^{k,l}_w(\l_0,\eta)| \leq M \|\eta\|(1+|\l|)^d(1+t)e^{(|\Re \l_0| -\rho)t}.
\end{equation}

\noindent{\bf Case 1. $-\rho < \Re \l_0 < \rho$:} Using the above estimates we get that $F_w^{k,l}(\l_0,\eta)(t)$ is bounded in $t>0$.\\

{\bf Case 2. $|\Re \l_0| = \rho$:}
From asymptotics (\ref{eq:a}), we have 
\[|F_w^{k,l}(\l_0,\eta)| \leq M_{\l_0,\eta_w}e^{(|\Re \l_0| -\rho)}t\]
Thus, if $|\Re{\l_0}| = \rho$, then $F_w^{k,l}(\l_0,\eta)$ is bounded for $t>0$.\\

\noindent{\bf Case 3. $|\Re \l_0|>\rho$:} Let $\l_0 \notin A$ and $|\Re \l_0|>\rho$ with $\l_0 \notin -\mathbb{N}_0$.  Then, $p(\l_0) \neq 0$ and $p_R(\l_0)C^{k,l}(\l_0)\neq0$. Thus,
we have that 
\[|F^{k,l}_w(\l_0,\eta)(t)| \asymp |\eta_w |e^{(|\Re \l_0| -\rho)t},\]
\[|F^{k,l}_w(\l_0,\eta)(t)| \asymp |\eta_w p_R^{k,l}(\l_0)|e^{(|\Re \l_0| -\rho)t},\]
implying that it is unbounded.\\
\noindent{\bf Case 4: $\l_0 \in -\mathbb{N}_0$}. In this case, from Lemma~\ref{lemma:singktype} we have that $p_R^{k,l}(\l_0)\eE_{k,l}(\l_0,\eta)$ is identically zero for $t>0$ when $\l_0 \in -\bN_0$. \\
\end{proof}

 As a consequence of Helgason-Johnson theorem and the estimate (\ref{eq:k,l}) it follows that the Fourier transform $\l \mapsto \cF^{k,l}f(\l)\eta$ of $f \in L^1(X)^{k,l}$ extends meromorphically to $S_1$ and $p_R^{k,l}(\lambda)\cF^{k,l}f(\l)\eta$ is holomorphic in the interior $S_1^\circ \cap a^*(R)$. It also follows that if $\phi \in C_c^\infty(X)^{k,l}$, then there exists a constant $M$  such that for any $n \in \mathbb{N}_0$
\[|p(-\l)\cF^{k,l} \phi(\l)(\eta)| \leq M (1+|\l|)^{-n}e^{r|\Re \l|}\|\eta\|, \quad \l \in a^*(R)\]
where $r>0$ depends on the support of $\phi$.
Furthermore, we have the following theorems and consequences:
\begin{theorem}\label{thm:HYktype}
    Let $\l_0  \in S_r \cap a^*(R)$ be a real number.  Then
	\[\left( \int_{ \mathbb R}\left|\frac{p_R(-(\l_0 + i\l))\cF^{k,l} f(\l_0 + i\l)\eta}{(1+\l_0 + \l)^d}\right|^{r'} d\l \right)^{1/r'} \leq M \|f\|_r \|\eta\|, \quad {f \in L^r(X)^{k,l}}, \]
    for some $M,d > 0$ and for any $\eta \in \C{}^{\cW}$.
\end{theorem}

\begin{theorem}
     Let $1 \leq r <2$, $\theta = 2/r - 1$ and $\l_0 + i\l \in S_r^\circ \cap a^*(R)$. Then for any $f \in L^r(X)^{k,l}$,
    \[\lim_{|\l|>0, |\l| \rightarrow \infty} |\cF^{k,l} f(\l_0 + i\l)\eta| = 0.\]
\end{theorem}

The proof of the above theorems follows similarly as in the case of $K$-invariant functions. 

\begin{remark}
It is known that, for $\alpha,\beta \geq -\frac12$, the Jacobi function
$\varphi^{(\alpha,\beta)}_{-i\lambda}(t)$ is bounded as a function of $t$
whenever $\lambda \in S_1$. Moreover, when
$-\frac12 \leq \alpha \leq \beta$, this condition is also necessary; that is,
$\varphi^{(\alpha,\beta)}_{-i\lambda}(t)$ is bounded in $t$ if and only if
$\lambda \in S_1$. In the pseudo-Riemannian real hyperbolic case
$SO_e(p,q)/SO_e(p-1,q)$, the left $K$-invariant Eisenstein integral is given by
\[
\eE_w(\lambda,\eta)(t)
=
\eta_w\,2^{\lambda-\rho}
\frac{
\Gamma\left(\frac{\lambda+\rho}{2}\right)
\Gamma\left(\frac{\lambda-\rho+q}{2}\right)
}{
\Gamma(\lambda)\Gamma\left(\frac{q}{2}\right)
}
\varphi_{-i\lambda}^{(\alpha,\beta)}(t),
\]
where
\[
\alpha=\frac q2-1,
\qquad
\beta=\frac p2-1,
\]
with $p>1$, $q\geq 1$, or $p\geq 1$, $q>1$. By the Helgason--Johnson theorem \ref{thm:ktypehj}, the Eisenstein integrals
are unbounded in $t$ whenever $\operatorname{Re}\lambda>\rho$. Since
$\eE_w(\lambda,\eta)$ has no pole in this region, the above formula implies
that the Jacobi function
$\varphi^{(\alpha,\beta)}_{-i\lambda}(t)$ is also unbounded in $t$ for
$\Re\lambda>\rho$. Finally, since Jacobi functions are even in
the spectral parameter $\lambda$, the same conclusion holds for
$\Re\lambda<-\rho$. Consequently, for
$\alpha\geq 0,\quad \beta\geq -\frac12,$ or $
\alpha\geq -\frac12,\quad \beta\geq 0,$
the Jacobi function $\varphi^{(\alpha,\beta)}_{-i\lambda}(t)$ is bounded in
$t$ if and only if
\[
|\Re\lambda|\leq \rho.
\] 
Thereby extending the known boundedness result of Flensted-Jensen and Koornwinder \cite{FT73} for Jacobi functions to the above parameter range for integers and half integers $\alpha$ and $\beta$.
\end{remark}

\section*{Acknowledgement}We are thankful to the anonymous referee for many valuable suggestions, which have improved the article.

\bibliographystyle{plain} 
\bibliography{refs}

\end{document}